\documentclass[12pt]{amsart}
\usepackage{amssymb}
\usepackage{mathtools}
\usepackage{txfonts}
\usepackage{eucal}
\usepackage{extarrows}
\usepackage{color}

\usepackage{color}

\usepackage[all]{xy}

\usepackage{tikz}

\usepackage{graphicx}
\usepackage{float}

\usepackage{xspace}

\usepackage[a4paper,body={16.3cm,22.8cm},centering]{geometry}

\ifpdf
\usepackage[colorlinks,final,backref=page,hyperindex]{hyperref}
\else
\usepackage[colorlinks,final,backref=page,hyperindex,hypertex]{hyperref}
\fi
\usepackage[active]{srcltx} 

\newcommand{\nc}{\newcommand}
\newcommand{\delete}[1]{}

\nc{\mlabel}[1]{\label{#1}}  
\nc{\mcite}[1]{\cite{#1}}  
\nc{\mref}[1]{\ref{#1}}  
\nc{\meqref}[1]{\eqref{#1}}  
\nc{\mbibitem}[1]{\bibitem{#1}} 

\delete{
\nc{\mlabel}[1]{\label{#1}  
{\hfill \hspace{1cm}{\small\tt{{\ }\hfill(#1)}}}}
\nc{\mcite}[1]{\cite{#1}{\small{\tt{{\ }(#1)}}}}  
\nc{\mref}[1]{\ref{#1}{{\tt{{\ }(#1)}}}}  
\nc{\meqref}[1]{\eqref{#1}{{\tt{{\ }(#1)}}}}  
\nc{\mbibitem}[1]{\bibitem[\bf #1]{#1}} 
}

\newtheorem{theorem}{Theorem}[section]
\newtheorem{prop}[theorem]{Proposition}
\newtheorem{lemma}[theorem]{Lemma}
\newtheorem{coro}[theorem]{Corollary}

\theoremstyle{definition}
\newtheorem{defn}[theorem]{Definition}
\newtheorem{remark}[theorem]{Remark}
\newtheorem{exam}[theorem]{Example}
\newtheorem{prop-def}{Proposition-Definition}[section]

\newcommand\alphlist{a,b,c,d,e,f,g,h,i,j,k,l,m,n,o,p,q,r,s,t,u,v,w,x,y,z}
\newcommand\Alphlist{A,B,C,D,E,F,G,H,I,J,K,L,M,N,O,P,Q,R,S,T,U,V,W,X,Y,Z}
\newcommand\getcmds[3]{\expandafter\newcommand\csname #2#1\endcsname{#3{#1}}}
\makeatletter
\@for\x:=\alphlist\do{\expandafter\getcmds\expandafter{\x}{frak}{\mathfrak}}
\@for\x:=\Alphlist\do{\expandafter\getcmds\expandafter{\x}{frak}{\mathfrak}}
\makeatother

\nc{\name}[1]{{\bf #1}}

\nc{\bfk}{{\bf k}}
\font\cyr=wncyr10

\newfont{\scyr}{wncyr10 scaled 550}
\nc{\sha}{\mbox{\cyr X}}
\nc{\ssha}{\mbox{\bf \scyr X}}

\nc{\Id}{\mathrm{Id}}
\nc{\lbar}[1]{\overline{#1}}

\nc{\leaf}{\mathrm{leaf}}

\usetikzlibrary{calc}

\nc{\fax}{\mathcal{F}(X)} 
\nc{\pfa}{\mathcal{PF}}
\nc{\shpr}{\diamond}    
\nc{\ot}{\otimes}      
\nc{\dep}{\mathrm{dep}} 
\nc{\var}{\varepsilon} 
\nc{\id}{\mathrm{id}}  
\nc{\set}{\mathbf{Set}} 
\nc{\vect}{\mathbf{Vect}} 
\nc{\spep}{\mathbf{p}} 
\nc{\speq}{\mathbf{q}} 
\nc{\spei}{\mathbf{1}_{\bfk}} 
\nc{\Hom}{\mathrm{Hom}}
\nc{\calc}{\mathcal{C}}
\nc{\adf}{\mathrm{ADF}}
\nc{\pure}{simple\xspace}
\nc{\Coinv}{\mathrm{Coinv}}
\nc{\ff}{\mathcal{F}}
\nc{\bff}{\widehat{\mathcal{F}}}
\nc{\com}{\mathrm{Com}}
\nc{\comp}{\mathrm{Comp}}
\nc{\rbs}{\text{Rota-Baxter species}}

\renewcommand{\geq}{\geqslant}
\renewcommand{\leq}{\leqslant}

\nc{\li}[1]{\textcolor{red}{#1}}
\nc{\lir}[1]{\textcolor{red}{Li:#1}}
\nc{\peng}[1]{\textcolor{purple}{Peng:#1}}
\nc{\yi}[1]{\textcolor{cyan}{Yi:#1}}
\nc{\revise}[1]{\textcolor{red}{#1}}

\nc{\hrtb}{\mathcal{H}_{RT}(X\sqcup\Omega)} \nc{\hrts}{\mathcal{H}_{\mathrm{RT}}(X, \Omega)}\nc{\rts}{\mathcal{T}(X, \Omega)}\nc{\rfs}{\mathcal{F}(X, \Omega)} \nc{\counit}{\varepsilon_{\mathrm{RT}}}

\nc{\Po}{(P_\omega)_{\omega\in \Omega}}
\nc{\Pop}{(P'_\omega)_{\omega\in \Omega}}
\nc{\Bo}{(B_{\omega}^+)_{\omega\in \Omega}}
\nc{\col}{\Delta}
\nc{\coll}{\Delta_{\lambda}}
\nc{\calt}{{\mathcal T}}
\newcommand{\sh}{\mathrm{Sh}}
\newcommand{\qsh}{\mathrm{QSh}}
\nc{\bre}{{\rm bre}}
\nc{\etree}{1}
\nc{\calf}{{\mathcal F}}
 \nc{\conc}{m_{RT}}
 \nc{\RT}{\mathrm{RT}}
 \nc{\mul}{m_{\mathrm{RT}}}
\nc{\free}[1]{\bar{#1}}
\nc{\hck}{H_{RT}}

\newcommand{\tdun}[1]
{\begin{picture}(10,5)(-2,-1)
\put(0,0){\circle*{2}}
\put(3,-2){\tiny #1}
\end{picture}}

\newcommand{\tddeux}[2]{\begin{picture}(12,5)(0,-1)
\put(3,0){\circle*{2}}
\put(3,0){\line(0,1){5}}
\put(3,5){\circle*{2}}
\put(6,-3){\tiny #1}
\put(6,3){\tiny #2}
\end{picture}}

\newcommand{\tdtroisun}[3]{\begin{picture}(20,12)(-5,-1)
\put(3,0){\circle*{2}}
\put(-0.65,0){$\vee$}
\put(6,7){\circle*{2}}
\put(0,7){\circle*{2}}
\put(5,-2){\tiny #1}
\put(8,5){\tiny #2}
\put(-6,5){\tiny #3}
\end{picture}}
\newcommand{\tdtroisdeux}[3]{\begin{picture}(12,12)(-2,-1)
\put(0,0){\circle*{2}}
\put(0,0){\line(0,1){5}}
\put(0,5){\circle*{2}}
\put(0,5){\line(0,1){5}}
\put(0,10){\circle*{2}}
\put(3,-2){\tiny #1}
\put(3,3){\tiny #2}
\put(3,9){\tiny #3}
\end{picture}}

\newcommand{\tdquatreun}[4]{\begin{picture}(20,12)(-5,-1)
\put(3,0){\circle*{2}}
\put(-0.6,0){$\vee$}
\put(6,7){\circle*{2}}
\put(0,7){\circle*{2}}
\put(3,7){\circle*{2}}
\put(3,0){\line(0,1){7}}
\put(5,-2){\tiny #1}
\put(8.5,5){\tiny #2}
\put(1,10){\tiny #3}
\put(-5,5){\tiny #4}
\end{picture}}
\newcommand{\tdquatredeux}[4]{\begin{picture}(20,20)(-5,-1)
\put(3,0){\circle*{2}}
\put(-.65,0){$\vee$}
\put(6,7){\circle*{2}}
\put(0,7){\circle*{2}}
\put(0,14){\circle*{2}}
\put(0,7){\line(0,1){7}}
\put(5,-2){\tiny #1}
\put(9,5){\tiny #2}
\put(-6,5){\tiny #3}
\put(-6,12){\tiny #4}
\end{picture}}
\newcommand{\tdquatretrois}[4]{\begin{picture}(20,20)(-5,-1)
\put(3,0){\circle*{2}}
\put(-.65,0){$\vee$}
\put(6,7){\circle*{2}}
\put(0,7){\circle*{2}}
\put(6,14){\circle*{2}}
\put(6,7){\line(0,1){7}}
\put(5,-2){\tiny #1}
\put(8,5){\tiny #2}
\put(-6,5){\tiny #4}
\put(8,12){\tiny #3}
\end{picture}}
\newcommand{\tdquatrequatre}[4]{\begin{picture}(20,14)(-5,-1)
\put(3,5){\circle*{2}}
\put(-.65,5){$\vee$}
\put(6,12){\circle*{2}}
\put(0,12){\circle*{2}}
\put(3,0){\circle*{2}}
\put(3,0){\line(0,1){5}}
\put(6,-3){\tiny #1}
\put(6,4){\tiny #2}
\put(9,12){\tiny #3}
\put(-5,12){\tiny #4}
\end{picture}}
\newcommand{\tdquatrecinq}[4]{\begin{picture}(12,19)(-2,-1)
\put(0,0){\circle*{2}}
\put(0,0){\line(0,1){5}}
\put(0,5){\circle*{2}}
\put(0,5){\line(0,1){5}}
\put(0,10){\circle*{2}}
\put(0,10){\line(0,1){5}}
\put(0,15){\circle*{2}}
\put(3,-2){\tiny #1}
\put(3,3){\tiny #2}
\put(3,9){\tiny #3}
\put(3,14){\tiny #4}
\end{picture}}

\begin{document}

\title[Moerdijk Hopf algebras of decorated rooted forests]{Moerdijk Hopf algebras of decorated rooted forests: an operated algebra approach}
%

\author{Lo\"\i c Foissy}
\address{Univ. Littoral C\^ote d'Opale, UR 2597 LMPA, Laboratoire de Math\'ematiques Pures et Appliqu\'ees Joseph Liouville F-62100 Calais, France}
\email{loic.foissy@univ-littoral.fr}

\author{Xiao-song Peng}
\address{School of Mathematics and Statistics,
Jiangsu Normal University, Xuzhou, Jiangsu 221116, P.\,R. China}
\email{pengxiaosong3@163.com}

\author{Yunzhou Xie}
\address{Department of Mathematics, Imperial College London, London SW7 2AZ, UK}
\email{yx3021@ic.ac.uk}

\author{Yi Zhang}
\address{School of Mathematics and Statistics,
	Nanjing University of Information Science \& Technology, Nanjing, Jiangsu 210044, P.\,R. China}
\email{zhangy2016@nuist.edu.cn}

\date{\today}
\begin{abstract}
In this paper, we first endow the space of decorated planar rooted forests with a coproduct that equips it with the structure of a bialgebra and further a Moerdijk Hopf algebra. We also present a combinatorial description of this coproduct, and further give an explicit formulation of its dual coproducts through the newly defined notion of forest-representable matrices.
By viewing the Moerdijk Hopf algebra within the framework of operated algebras, we introduce the notion of a multiple cocycle Hopf algebra, incorporating a symmetric Hochschild 1-cocycle condition. We then show that the antipode of this Hopf algebra is a  Rota-Baxter operator on Moerdijk Hopf algebras.
Furthermore, we investigate the universal properties of cocycle Hopf algebras. As an application, we construct the initial object in the category of free cocycle Hopf algebras on undecorated planar rooted forests, which coincides with the well-known Moerdijk Hopf algebra.
\end{abstract}

\subjclass[2010]{
16W99, 
05C05, 
16S10, 
16T10, 
16T30,  
17B60, 
}

\keywords{Rooted forest; Hopf algebra; Cocycle condition}

\maketitle

\tableofcontents

\setcounter{section}{0}

\allowdisplaybreaks

\section{Introduction}

Rooted tree Hopf algebras, originally introduced by Connes and Kreimer in the study of renormalization, capture the recursive structure of Feynman diagrams in an algebraic framework. In this paper, we construct a Moerdijk Hopf algebra on the space of decorated planar rooted forests and characterize its coproduct combinatorially using admissible cuts. Within the setting of operated algebras, we introduce multiple cocycle Hopf algebras and prove that the constructed Moerdijk Hopf algebra is the free object in this category.

\subsection{Rooted tree bialgebras and Hopf algebras}

A rooted tree Hopf algebra is an algebraic structure that arises in the combinatorial study of rooted trees, playing a crucial role in quantum field theory~\mcite{BF10, CK98,CGPZ20, Kre98}, noncommutative geometry~\mcite{GPZ11}, and numerical analysis~\mcite{Bro04}.
Such an algebra employs rooted trees as basis elements, with the product, coproduct, and antipode operations encoding their combinatorial and recursive properties.

The study of Hopf algebras in combinatorics and mathematical physics has gained significant attention due to their rich algebraic structures and broad applications.
The foundational work of Connes and Kreimer ~\mcite{CK98} established the connection between renormalization and combinatorial Hopf algebras, using rooted trees as a key combinatorial structure to model the recursive nature of Feynman diagram subtractions. Since then, rooted tree Hopf algebras have been extensively studied in various contexts, including operads~\mcite{CL01}, differential algebras~\mcite{GL05}, Rota-Baxter algebras~\mcite{ZGG16, ZXG}, pre-Lie and Lie algebra~\mcite{Mur06} and control theory.

Recent advancements in this field focus on the structural properties and applications of rooted tree Hopf algebras. Researchers have explored their bialgebraic structure, graded connections, and cohomological interpretations. Building upon this foundation, researchers have progressively expanded the algebraic representation framework through novel Hopf algebra constructions, including the Foissy-Holtkamp \mcite{Foi02, Hol03}, Grossman-Larson~\mcite{GL89},  Loday-Ronco ~\mcite{LR98}, Moerdijk~\mcite{Moe01}.

Applications have expanded into numerical methods, particularly Butcher's B-series in numerical integration~\mcite{LM11}, where the Hopf algebraic structure encodes the composition of numerical schemes. Furthermore, these algebras have found applications in stochastic processes~\mcite{BHZ19}, perturbative expansions~\mcite{CH24}, and even the Minimalist Program in generative linguistics~\mcite{MBC23}
where Hopf algebraic structures play a key role. Current research trends involve deeper categorical approaches~\mcite{Koc13}, interactions with pre-Lie algebras~\mcite{CL01}, and explicit computations of cohomology groups in tree-related Hopf algebras~\mcite{Foi18}.

A landmark advancement occurred in 2019 when Bruned, Hairer, and Zambotti~\mcite{BHZ19} established a Hopf algebra on decorated rooted forests with types (which are decorations on the edges) to algebraically characterize renormalization procedures in stochastic partial differential equations. This breakthrough propelled rooted forest combinatorics to the forefront of mathematical physics research. Subsequent developments by Foissy~\mcite{Foi18} generalized this framework by constructing extended Connes-Kreimer-type algebras.
Recent studies in~\mcite{FGPXZ, PZGL, ZXG, ZCGL19, ZL25} systematically investigated infinitesimal bialgebras, twisted bialgebras, (left counit) Hopf algebras on decorated rooted forests(trees), demonstrating sustained theoretical innovation in this domain.

In this paper, we construct a coproduct on decorated planar rooted forests by employing a symmetric 1-cocycle condition, thereby highlighting its combinatorial nature. A detailed combinatorial characterization of this coproduct is established. We also explicitly describe its dual coproducts through the newly introduced notion of forest-representable matrices. We  represent every decorated forest $F$ in the space of all decorated forests $\rfs$ by a matrix $M(F)$ of size $n\times(n+1)$.
Then we obtain the following one of our main results about the dual product $\star$.
\begin{theorem}[Theorem~\mref{thm:dul}]
Let $F,G\in \rfs$, with respectively $k$ and $l$ vertices.
\begin{align*}
F\star G&=\sum_{\sigma \in \sh(k,l)} \sum_{C\in \mathcal{FM}_\sigma(M(F),M(G))}M^{-1}(C).
\end{align*}
\end{theorem}

These advancements highlight the construction of  bialgebras or even Hopf algebras on variously decorated rooted forests as a promising research direction. Particularly noteworthy is the combinatorial interpretation that rooted tree or forest structures provide for the coproduct operations in these bialgebras. Such combinatorial-algebraic correspondences not only drive theoretical developments but also serve as critical interfaces for practical applications.

\subsection{Operated Hopf algebras}
Inspired by theory of multi-operator groups~\mcite{Hig56}, Kurosh~\mcite{Kur60} pioneered the concept of algebras with linear operators. However, this theoretical framework initially received limited attention until its significance was rediscovered by Guo~\mcite{Guo09}, see also~\mcite{BCQ10}.

Utilizing combinatorial objects such as Motzkin paths, rooted forests, and bracketed words, Guo~\mcite{Guo09}  constructed free objects for these algebras. These algebraic structures are now collectively referred to as $\Omega$-operated algebras (or multi-operated algebras), where $\Omega$ denotes the index set of linear operators.
Of particular interest is the fact that the Connes-Kreimer Hopf algebra on rooted forests, when equipped with the grafting operator, naturally constitutes a special case of an operated algebra.

In 2016, Zhang-Gao-Guo~\mcite{ZGG16} innovatively introduced the concept of operated Hopf algebras by integrating the theories of Hopf algebras and operated algebras on decorated rooted forests, systematically investigating free objects in this category. Building on the Grobner-Shirshov method, the results were further generalized in~\mcite{ZGG22}, leading to the establishment of the  $\Omega$-operated Hopf algebra.

We would like to emphasize that the Loday-Ronco Hopf algebra can also be studied within the unified framework of operated algebras. In ~\mcite{ZG20}, Zhang and Gao introduced the
concepts of $\vee_\Omega$-Hopf algebra and proved that the Loday-Ronco Hopf algebra can be characterized as the free multiple 1-cocycle $\vee_\Omega$-Hopf algebra generated by the empty set. Later, Marcolli, Berwick, and Chomsky~\mcite{MBC23, MBC25} investigated applications of Hopf algebraic structures in Chomsky's Minimalist Program, a framework introduced in the 1990s to capture the most economical principles of language. They formalized syntactic merge as algebraic operations within Hopf algebras, examined the interplay of external and internal merge in the Loday-Ronco Hopf algebra, and analyzed its computational properties. Drawing inspiration from renormalization methods in physics, they further used operated Hopf algebras to model the computation of semantic meaning from syntactic expressions. This algebraic perspective not only strengthens the mathematical foundations of generative linguistics but also suggests new methods for natural language processing and computational linguistics.

%
%
%

These developments are not only of fundamental importance in algebraic combinatorics and computational mathematics but also exhibit broad applications in quantum field theory, renormalization theory, generative linguistics, and noncommutative geometry.

Considering that the Moerdijk Hopf algebra is another important class of Hopf algebras on rooted trees, its comultiplication structure is also related to the 1-cocycle condition. We will investigate the Moerdijk Hopf algebra from the perspective of operated algebras, with the expectation that it may also provide a new algebraic interpretation of the Minimalist Program, enabling a mathematically precise formulation of language generation rules.

In this paper, we study the Moerdijk Hopf algebra within the framework of operated algebras, we introduce the notion of a multiple cocycle Hopf algebra, incorporating a symmetric Hochschild 1-cocycle condition. Furthermore, we investigate the universal properties of cocycle Hopf algebras. As an application, we use the universal properties to get a cocycle bialgebra morphism.
\begin{coro}[Corollary~\mref{coro:uni}]
Let $\lambda,\mu \in \bfk$.
The following map is a bialgebra morphism:
\begin{align*}
\phi_\lambda:\left\{\begin{array}{rcl}
(\hrts, m_{\mathrm{RF}}, \Delta_\mu)&\longrightarrow&(\hrts, m_{\mathrm{RF}}, \Delta_{\lambda\mu})\\
F&\longmapsto&\lambda^{d_X(F)}F,
\end{array}\right.
\end{align*}
where $d_X(F)$ is the number of leaves of $F$ decorated by an element of $X$.
\end{coro}

\subsection{Rota-Baxter operators on Hopf algebras}
Lie groups and Lie algebras are fundamental concepts in algebra. Recently, Lang, Sheng, and Guo~\mcite{GLS21} introduced the notion of Rota-Baxter Lie groups and established that the differentiation of a Rota-Baxter Lie group yields a Rota-Baxter Lie algebra of weight 1. This result generalizes the classical fact that the differentiation of a Lie group gives rise to a Lie algebra. 

Lie algebras and Lie groups can be regarded as fundamental examples of cocommutative Hopf algebras. In~\mcite{Gon21}, Goncharov combined the concept of Rota-Baxter operators of weight 1 on Lie algebras with that of Rota-Baxter operators on groups and provided the definition of a Rota-Baxter operator of weight 1 on cocommutative Hopf algebras.

In this paper, we prove that the antipode of the Moerdijk Hopf algebra on decorated rooted forests is a Rota-Baxter operator. This discovery enriches the examples related to commutative Hopf algebras.

{\bf Structure of the Paper.}
In Section~\mref{sec:ibw}, we first recall some basic definitions and facts on decorated rooted trees and forests. By the symmetric 1-cocycle condition (Eq.~(\mref{eq:scocycle})), for a fixed $\lambda \in {\bf k}$, we define a coproduct $\coll$ on decorated planar rooted forests $\hrts$ to equip it with a new coalgebra structure (Theorem~\mref{thm:coalgebra}) and further a bialgebra structure (Theorem~\mref{thm:main}). Then we give a combinatorial description of the coproduct $\coll$ (Theorem~\mref{thm:comcoproduct}). As the coproduct $\Delta=\Delta_0$ is homogeneous, it induces a product $\star$ on $\hrts$. Using the concept of forest-representable matrices, we also give a description of the product $\star$ (Theorem~\mref{thm:dul}).

In Section~\mref{sec:idc}, a Moerdijk  Hopf algebra on decorated planar rooted forests is constructed for the case of $\lambda=0$ (Theorem~\mref{cor:hopf}) and a combinatorial description of the antipode is given (Theorem~\mref{thm:comantipode}). Since the constructed Hopf algebra is cocommutative, the antipode is a Rota-Baxter operator on the Hopf algebra (Proposition~\mref{prop:comantipode}).

In Section~\mref{sec:icw}, we give the concept of $\Omega$-cocycle bialgebras (Hopf algebras) and show the constructed one is a free object in such category (Theorem~\mref{thm:freecocycle}). As an application, we use the universal properties to get a bialgebra morphism between different $\Omega$-cocycle  bialgebras (Corollary~\mref{coro:uni}).

{\bf Notation.}
Throughout this paper, let $\bfk$ be a unitary commutative ring unless the contrary is specified,
which will be the base ring of all modules, algebras, coalgebras, bialgebras, tensor products, as well as linear maps.

\section{Bialgebras of decorated planar rooted forests}\label{sec:ibw}

In this section, we  construct a coproduct on decorated planar rooted forests based on a symmetric 1-cocycle condition, thereby equipping the space with a bialgebra structure. A combinatorial characterization of this coproduct is provided, and the dual coproducts are described explicitly via the newly introduced concept of forest-representable matrices.

\subsection{Decorated planar rooted forests}\mlabel{sucsec:deco}

In this subsection, we recall some basic definitions and facts on decorated rooted trees and forests that will be used in this paper, see ~\mcite{Foi02,  Guo09, PZGL, Sta97} for more details.

\begin{defn}
A {\bf rooted tree} consists of a connected, acyclic collection of vertices and directed edges, with a uniquely designated vertex called the {\bf root}. A {\bf planar rooted tree} is a rooted tree equipped with a fixed planar embedding that determines the order of branches.
\end{defn}

%

Denote by $\calt$ the set of all planar rooted trees. Let $M(\calt)$ be the free monoid generated by $\calt$ under the operation of concatenation. The identity element of $M(\calt)$, referred to as the {\bf empty tree}, is denoted by $1$. Any non-identity element of $M(\calt)$, called a {\bf planar rooted forest}, is a noncommutative concatenation of planar rooted trees, and can be expressed as $F = T_1 \cdots T_n$ with $T_1, \ldots, T_n \in \calt$. We adopt the convention that $F=1$ when $n=0$.

Given a nonempty set $\Omega$ and a set $X$ satisfying $X \cap \Omega = \emptyset$, we denote by $\rts$ the set of planar rooted trees whose internal vertices are decorated exclusively by elements of $\Omega$, while the leaves are decorated by elements from $X \sqcup \Omega$. Define $\rfs = M(\rts)$ to be the free monoid generated by $\rts$ with concatenation as the product. Elements of $\rfs$ are referred to as {\bf decorated planar rooted forests}.

\begin{exam}
The following are some examples in $\rts$:
$$\tdun{$\alpha$},\ \, \tdun{$x$},\ \, \tddeux{$\alpha$}{$\beta$},\ \,  \tddeux{$\alpha$}{$x$}, \ \, \tdtroisun{$\alpha$}{$\beta$}{$\gamma$},\ \,\tdtroisun{$\alpha$}{$x$}{$\gamma$}, \ \,\tdtroisun{$\alpha$}{$x$}{$y$}, \ \, \tdquatretrois{$\alpha$}{$\beta$}{$\gamma$}{$\beta$},\ \, \tdquatretrois{$\alpha$}{$\beta$}{$\gamma$}{$x$}, \ \, \tdquatretrois{$\alpha$}{$\beta$}{$x$}{$y$},$$
with $\alpha,\beta,\gamma\in \Omega$ and $x, y \in X$.
\end{exam}

For $F = T_1 \cdots T_n \in \rfs$ with $n \geq 0$ and $T_1, \ldots, T_n \in \rts$, we define $\bre(F) := n$ to be the {\bf breadth} of $F$, that is, the number of trees in the forest. In particular, we adopt the convention that $\bre(\etree) = 0$ when $n = 0$.

Let $\bullet_{X} := \{\bullet_{x} \mid x \in X\}$, and define
\begin{align*}
\calf_0 := M(\bullet_{X}) = S(\bullet_{X}) \sqcup \{\etree\},
\end{align*}
where $M(\bullet_{X})$ (respectively, $S(\bullet_{X})$) denotes the submonoid (respectively, subsemigroup) of $\rfs$ generated by $\bullet_X$.

Assuming that $\calf_n$ has already been constructed for some $n \geq 0$, we define the next level recursively as
\begin{align*}
\calf_{n+1} := M\left(\bullet_{X} \sqcup \left(\bigsqcup_{\omega \in \Omega} B_{\omega}^{+}(\calf_n)\right)\right).
\end{align*}
This yields an ascending chain $\calf_n \subseteq \calf_{n+1}$, and hence
\begin{align*}
\rfs = \varinjlim \calf_n = \bigcup_{n=0}^{\infty} \calf_n.
\end{align*}
An element $F \in \calf_n \setminus \calf_{n-1}$ is said to have {\bf depth} $n$, and we denote this by $\dep(F) = n$.

%

\begin{exam}
Here are some examples about the depths of some decorated planar rooted forests.
\begin{align*}
\dep(\etree) =&\ \dep(\bullet_x) =0,\ \dep(\bullet_\omega)=\dep(B^+_{\omega}(\etree)) = 1,\  \dep(\tddeux{$\omega$}{$\alpha$})= \dep(B^+_{\omega}(B^+_{\alpha}(\etree))) =2, \\
\dep(\tdun{$x$}\tddeux{$\omega$}{$y$}\tdun{$y$}) =&\ \dep(\tddeux{$\omega$}{$y$})=
\dep(B^+_{\omega}(\bullet_y)) =1, \ \dep(\tdtroisun{$\omega$}{$x$}{$\alpha$}) = \dep(B^+_{\omega}(B^+_{\alpha}(\etree) \bullet_x)) = 2,
\end{align*}
where $\alpha,\omega\in \Omega$ and $x, y \in X$.
\end{exam}

%

Define
\begin{align*}
\hrts:= \bfk \rfs=\bfk M(\rts)
\end{align*}
to be the {\bf free $\bfk$-module} spanned by $\rfs$. Then $\hrts$ becomes the free noncommutative algebra generated by the decorated planar rooted trees $\rts$ under the concatenation product, denoted by $m_{\mathrm{RT}}$ and typically omitted for simplicity. For any $F \in \rfs$, the {\bf weight} of $F$ is defined to be the total number of vertices in $F$. With this notion, we obtain a natural grading
$\hrts=\mathop{\oplus} \limits_{n \geq 0} \hrts_n$,
where each $\hrts_n$ is the $\bfk$-linear span of all forests of weight $n$, making $\hrts$ into a graded algebra.

For each $\omega\in \Omega$, define
$$B^+_\omega:\hrts\to \hrts$$
to be the linear {\bf grafting operation}, which maps the unit $1$ to the single-node tree $\bullet_\omega$, and sends a rooted forest in $\hrts$ to the tree obtained by grafting all components onto a new root decorated by $\omega$.

\begin{exam}
The following are some grafting operations:
\begin{align*}
B_{\omega}^{+}(\etree)&=\tdun{$\omega$} \ ,&  B_{\omega}^{+}(\tdun{$x$}\tddeux{$\alpha$}{$\beta$})&=\tdquatretrois{$\omega$}{$\alpha$}{$\beta$}{$x$},&  B_{\omega}^{+}(\tddeux{$\alpha$}{$y$}\tdun{$x$})&=
\tdquatredeux{$\omega$}{$x$}{$\alpha$}{$y$},
\end{align*}
where $\alpha, \beta, \omega\in \Omega$ and $x, y \in X$.
\end{exam}

%
%

\subsection{From Cartier-Quillen cohomologies to symmetric 1-cocycle conditions}

Given an algebra $A$ and a bimodule $M$ over $A$, let $H^{*}(A, M)$ denote the {\bf Hochschild cohomology} of $A$ with coefficients in $M$, which is defined via a cochain complex whose cochains are maps $A^{\ot n} \rightarrow M$. For further details, see~\mcite{Lod92}.

Let $(C, \Delta)$ be a coalgebra and $(B, \delta_{G}, \delta_{D})$ a bicomodule over $C$. The {\bf Cartier-Quillen cohomology} of $C$ with coefficients in $B$ serves as a dual counterpart to the Hochschild cohomology. Specifically, it is defined as the cohomology of the complex $\mathrm{Hom}{\mathbf{k}}(B, C^{\ot n})$, with differentials $b_n: \mathrm{Hom}{\mathbf{k}}(B, C^{\ot n})\rightarrow \mathrm{Hom}{\mathbf{k}}(B, C^{\ot (n+1)})$ given by
\begin{align*}
b_n(L)=(\id \ot L)\circ \delta_{G}+\sum_{i=1}^{n}(-1)^i(\id _{C}^{\ot{(i-1)}}\ot \Delta \ot \id _{C}^{\ot{(n-i)}})L+(-1)^{n+1}(L\ot \id)\circ \delta_D,
\end{align*}
where $L: B\rightarrow C^{\ot n}$. Particularly, a linear map $L: B \rightarrow C$ is called a {\bf 1-cocycle} if it satisfies the following condition:
\begin{align*}
\Delta\circ L= (L\ot \id)\circ \delta_D + (\id \ot L)\circ \delta_{G},
\end{align*}
see~\mcite{Fo3, Moe01} for more details.

We consider the special case where the bicomodule is given by $(C, \delta_{G}, \delta_D)$ with $\delta_{G}=\delta_{D}=\Delta$. In this setting, a 1-cocycle is a linear endomorphism $L$ on $C$ satisfying
\begin{align}
\Delta \circ L(x)=(L \ot \id)\circ \Delta(x)+(\id \ot L)\circ \Delta(x), \quad \text{ for } x\in C.
\mlabel{eq:scocycle}
\end{align}
We refer to Eq.~(\mref{eq:scocycle}) as the {\bf symmetric 1-cocycle condition}.

\begin{remark}
\begin{enumerate}
\item When $L = B^+$, the symmetric 1-cocycle condition in Eq.~(\mref{eq:scocycle}) takes the form
\begin{equation*}
\Delta(F) = \Delta B^{+}(\lbar{F}) = (B^{+} \otimes \id) \Delta(\lbar{F}) + (\id \otimes B^{+}) \Delta(\lbar{F}) \quad \text{for } F = B^{+}(\lbar{F}) \in \mathcal{F},
\end{equation*}
which differs from the classical 1-cocycle condition introduced and studied by Connes-Kreimer\mcite{CK98}, given by
\begin{equation*}
\Delta(F) = \Delta B^{+}(\lbar{F}) = B^{+}(\lbar{F}) \otimes \etree + (\id \otimes B^{+}) \Delta(\lbar{F}),
\end{equation*}
which corresponds to the bicomodule $(C,\delta'_G,\delta'_R)$ with
\begin{align*}
\delta_R(x)=x\otimes 1,\delta_L(x)&=\Delta(x),
\end{align*}
for any $x \in C$.
\item Let $\mathbb{P}$ be a Hopf operad, and let $H$ be the algebra generated by finite rooted forests equipped with a linear endomorphism $\varrho$. Define maps $\sigma_i : H \rightarrow H$, for $i=1, 2$, by setting $\sigma_i(F)=q_i^n F$
for any forest $F$ with $n$ vertices, where $q_1, q_2 \in \bfk$. Moerdijk~\mcite{Moe01} showed that the pair
\begin{align*}
(\sigma_1, \sigma_2) = \sigma_1 \otimes \varrho + \varrho \otimes \sigma_2 : H \otimes H \to H \otimes H
\end{align*}
induces a $\mathbb{P}[t]$-algebraic structure on $H \otimes H$, which in turn gives rise to a family of Hopf $\mathbb{P}$-algebra structures parameterized by $\sigma_i$, $i=1, 2$. Notably, the choice $q_1=0$ and $q_2=1$ recovers the classical 1-cocycle relation, while the case $q_1 = q_2 = 1$ corresponds to the symmetric 1-cocycle condition employed in the present work; see~\cite{LM06, Moe01} for further details.
\end{enumerate}

\end{remark}

\subsection{A bialgebra structure on decorated planar rooted forests}
For the rest of the paper, we assume that $\lambda \in \bf k$ is a fixed element. In this subsection, we shall construct a bialgebra structure on decorated planar rooted forests.

First we define a new coproduct $\coll$ on $\hrts$ by induction on depth. By linearity, we only need to define $\coll(F)$ for basis elements $F\in \rfs$.
For the initial step of $\dep(F)=0$, we  define
\begin{equation}
\coll(F) :=
\left\{
\begin{array}{ll}
\etree \ot \etree, & \text{ if } F = \etree, \\
\bullet_{x} \ot 1+1 \ot \bullet_{x}+\lambda \bullet_{x} \ot \bullet_{x}, & \text{ if } F = \bullet_x \text{ for some } x \in X,\\
\coll(\bullet_{x_1}) \cdots \coll(\bullet_{x_{m}}), & \text{ if }  F=\bullet_{x_{1}}\cdots \bullet_{x_{m}} \text{ with } m\geq 2 \text{ and } x_i \in X.
\end{array}
\right.
 \mlabel{eq:dele}
\end{equation}
For the induction step of $\dep(F)\geq 1$, we reduce the definition to induction on breadth.
If $\bre(F) = 1$, we
write $F=B_{\omega}^{+}(\lbar{F})$ for some $\omega\in \Omega$ and $\lbar{F}\in \rfs$, and define
\begin{equation}
\coll(F)=\coll B_{\omega}^{+}(\lbar{F}) := (B_{\omega}^{+}\otimes\id)\coll(\lbar{F}) + (\id\otimes B_{\omega}^{+})\coll(\lbar{F}).
\mlabel{eq:dbp}
\end{equation}
In other words
\begin{align}
\coll B_{\omega}^{+}=(B_{\omega}^{+} \otimes \id + \id\otimes B_{\omega}^{+})\coll.
\mlabel{eq:cdbp}
\end{align}
If $\bre(F) \geq 2$, we write $F=T_{1}T_{2}\cdots T_{m}$ with $m\geq 2$ and $T_1, \ldots, T_m \in \rts$, and define
\begin{equation}
\coll(F)=\coll(T_{1}) \cdots \coll(T_{m}).
\mlabel{eq:delee1}
\end{equation}

\begin{exam}\mlabel{exam:cop}
Let $x,y \in X$ and $\alpha, \beta, \gamma \in \Omega$. Then
\begin{align*}
\coll(\tdun{$\alpha$})=&\ \tdun{$\alpha$}\ot \etree+\etree \ot \tdun{$\alpha$},\\
\coll(\tddeux{$\alpha$}{$x$})=&\ \tddeux{$\alpha$}{$x$} \ot \etree+\tdun{$\alpha$} \ot \tdun{$x$}+\tdun{$x$} \ot \tdun{$\alpha$}+ \etree \ot \tddeux{$\alpha$}{$x$}+ \lambda \tddeux{$\alpha$}{$x$} \ot \tdun{$x$}+\lambda \tdun{$x$} \ot \tddeux{$\alpha$}{$x$},\\
\coll(\tdun{$x$}\tddeux{$\alpha$}{$y$})=&\ \tdun{$x$}\tddeux{$\alpha$}{$y$} \ot \etree+ \tddeux{$\alpha$}{$y$} \ot \tdun{$x$}
+\tdun{$x$}\tdun{$\alpha$} \ot \tdun{$y$}+\tdun{$\alpha$} \ot \tdun{$x$}\tdun{$y$} +\tdun{$x$}\tdun{$y$}\ot \tdun{$\alpha$}+ \tdun{$y$} \ot \tdun{$x$}\tdun{$\alpha$}\\
&\ +\tdun{$x$} \ot \tddeux{$\alpha$}{$y$}+ \etree \ot \tdun{$x$}\tddeux{$\alpha$}{$y$}+ \lambda \tdun{$x$}\tddeux{$\alpha$}{$y$} \ot \tdun{$x$}+ \lambda \tdun{$x$}\tdun{$\alpha$} \ot \tdun{$x$}\tdun{$y$}+\lambda \tddeux{$\alpha$}{$y$} \ot \tdun{$x$} \tdun{$y$}+\lambda \tdun{$x$}\tddeux{$\alpha$}{$y$} \ot \tdun{$y$}\\
&\ +\lambda^2 \tdun{$x$}\tddeux{$\alpha$}{$y$} \ot \tdun{$x$}\tdun{$y$}+ \lambda \tdun{$x$} \tdun{$y$} \ot \tdun{$x$} \tdun{$\alpha$}
+\lambda \tdun{$x$} \ot \tdun{$x$}\tddeux{$\alpha$}{$y$}+ \lambda \tdun{$x$} \tdun{$y$} \ot \tddeux{$\alpha$}{$y$}+ \lambda \tdun{$y$} \ot \tdun{$x$}\tddeux{$\alpha$}{$y$}+\lambda^2 \tdun{$x$}\tdun{$y$} \ot \tdun{$x$}\tddeux{$\alpha$}{$y$} ,\\
\coll( \tdtroisun{$\alpha$}{$\beta$}{$x$})=&\ \tdtroisun{$\alpha$}{$\beta$}{$x$} \ot \etree+ \tdun{$x$}\tdun{$\beta$} \ot \tdun{$\alpha$}+\tddeux{$\alpha$}{$\beta$} \ot \tdun{$x$}+\tdun{$\beta$} \ot \tddeux{$\alpha$}{$x$}+\tddeux{$\alpha$}{$x$} \ot \tdun{$\beta$}+\tdun{$x$} \ot \tddeux{$\alpha$}{$\beta$}+ \tdun{$\alpha$} \ot \tdun{$x$}\tdun{$\beta$}+ \etree \ot \tdtroisun{$\alpha$}{$\beta$}{$x$}\\
&\ +\lambda \tdtroisun{$\alpha$}{$\beta$}{$x$} \ot \tdun{$x$}+ \lambda \tdun{$x$} \tdun{$\beta$} \ot \tddeux{$\alpha$}{$x$}+ \lambda \tddeux{$\alpha$}{$x$} \ot \tdun{$x$}\tdun{$\beta$}+ \lambda \tdun{$x$} \ot \tdtroisun{$\alpha$}{$\beta$}{$x$}.
\end{align*}
\end{exam}

Then we record the following two lemmas as a preparation.

\begin{lemma}
Let $F_1, F_2\in \hrts$. Then $\coll(F_1F_2)=\coll(F_1)\coll(F_2)$.
\mlabel{lem:colmor}
\end{lemma}

\begin{proof}
It follows directly from the definition of $\col$.
\end{proof}

The following lemma shows that $\hrts$ is closed under the coproduct $\col$.

\begin{lemma} \label{lem:cclosed}
For $F\in \hrts$,
\begin{align*}
\coll(F)\in \hrts\ot\hrts.
\end{align*}
\end{lemma}

\begin{proof}
By the construction of decorated planar rooted forests, it is enough to show that $\coll(F)$ for basis elements $F\in \rfs$ is a sum of tensor products of decorated planar rooted forests whose internal vertices are not decorated by $X$.
Then this result follows by the definition of $\coll$ and the induction on $\dep(F)\geq 0$.
\end{proof}

Define a linear map $\counit: \hrts \rightarrow \bfk$ as follows: for $F \in \rfs$
\begin{equation}
\counit(F) :=
\left\{
\begin{array}{ll}
1_{\bfk}, & \text{ if } F = \etree, \\
0, & \text{ otherwise }  .
\end{array}
\right.
\mlabel{eq:counit}
\end{equation}

\begin{lemma}\label{lem:counitmor}
Let $F_1, F_2 \in \hrts$. Then
\begin{align}
\counit(F_1F_2)=\counit(F_1)\counit(F_2).
\mlabel{eq:counitmor}
\end{align}
\end{lemma}
\begin{proof}
By the linearity of $\varepsilon$, we only need to prove Eq.~(\mref{eq:counitmor}) holds for $F_1,F_2 \in \rfs$. If $F_1=\etree$ or $F_2=\etree$, without loss of generality, assume $F_1=\etree$. Then
\begin{align*}
\counit(F_1F_2)=\counit(F_2)=\counit(F_1)\counit(F_2).
\end{align*}
If $F_1 \neq \etree$ and $F_2 \neq \etree$, then $F_1F_2 \neq \etree$. Hence both sides of Eq.~(\mref{eq:counitmor}) are zero by Eq.~(\mref{eq:counit}), as desired.
\end{proof}

\begin{theorem}
The triple $(\hrts, \,\coll, \, \counit)$ is a coalgebra.
\mlabel{thm:coalgebra}
\end{theorem}

\begin{proof}
By the linearity of $\coll$ and $\counit$, we only need to prove the coassociativity of $\coll$ and the counity of $\counit$. Define
\begin{align*}
\mathcal{A}:= \{F \in \hrts \, \mid \, (\coll \ot \id) \coll(F)=(\id \ot \coll)\coll(F) \}.
\end{align*}
Note that $1 \in \mathcal{A}$ by $\coll(1)=1 \ot 1$. By Lemma~\mref{lem:colmor}, $\coll$ is an algebra homomorphism, hence both $(\coll \ot \id) \coll$ and $(\id \ot \coll) \coll$ are algebra homomorphisms and so $A$ is a subalgebra of $\hrts$. For any $x \in X$, we have
\begin{align*}
&\ (\coll \ot \id)\coll(\bullet_{x})\\
=&\ (\coll \ot \id)(\bullet_{x} \ot \etree + \etree \ot \bullet_{x}+\lambda \bullet_{x} \ot \bullet_{x})\\
=&\ \bullet_{x} \ot \etree \ot \etree+ \etree \ot \bullet_{x} \ot \etree+ \lambda \bullet_{x} \ot \bullet_{x} \ot \etree+ \etree \ot \etree \ot \bullet_{x}+ \lambda \bullet_{x} \ot \etree \ot \bullet_{x}+\lambda \etree \ot \bullet_{x} \ot \bullet_{x}\\
&\ + \lambda^2 \bullet_{x} \ot \bullet_{x} \ot \bullet_{x}\\
\end{align*}
and
\begin{align*}
&\ (\id \ot \coll)\coll(\bullet_{x})\\
=&\ (\id \ot \coll)(\bullet_{x} \ot \etree + \etree \ot \bullet_{x} + \lambda \bullet_{x} \ot \bullet_{x})\\
=&\ \bullet_{x} \ot \etree \ot \etree + \etree \ot \bullet_{x} \ot \etree + \etree \ot \etree \ot \bullet_{x} +\lambda \etree \ot \bullet_{x} \ot \bullet_{x}+ \lambda \bullet_{x} \ot \bullet_{x} \ot \etree+ \lambda \bullet_{x} \ot \etree \ot \bullet_{x} \\
&\ +\lambda^2 \bullet_{x} \ot \bullet_{x} \ot \bullet_{x},
\end{align*}
hence $\bullet_{x} \in \mathcal{A}$. For $F \in \mathcal{A}$ and $\omega \in \Omega$, we have
\begin{align*}
&\ (\coll \ot \id) \coll(B_{\omega}^{+}(F))\\
=&\ (\coll \ot \id)(B_{\omega}^{+} \ot \id + \id \ot B_{\omega}^{+})\coll(F)\\
=&\ (\coll B_{\omega}^{+} \ot \id+ \coll \ot B_{\omega}^{+})\coll(F)\\
=&\ (B_{\omega}^{+} \ot \id \ot \id+ \id \ot B_{\omega}^{+} \ot \id+ \id \ot \id \ot B_{\omega}^{+})(\coll \ot \id)\coll(F)
\end{align*}
and
\begin{align*}
&\ (\id \ot \coll) \coll(B_{\omega}^{+}(F))\\
=&\ (\id \ot \coll)(B_{\omega}^{+} \ot \id + \id \ot B_{\omega}^{+})\coll(F)\\
=&\ (B_{\omega}^{+} \ot \coll+ \id \ot \coll B_{\omega}^{+})\col(F)\\
=&\ (B_{\omega}^{+} \ot \id \ot \id+\id \ot B_{\omega}^{+} \ot \id + \id \ot \id \ot B_{\omega}^{+})(\id \ot \coll)\coll(F).
\end{align*}
Since $F \in \mathcal{A}$, $(\coll \ot \id)\coll(F)=(\id \ot \coll)\coll(F)$ and so
\begin{align*}
(\coll \ot \id)\coll(B_{\omega}^{+}(F))=(\id \ot \coll)\coll(B_{\omega}^{+}(F)),
\end{align*}
which means $B_{\omega}^{+}(F) \in \mathcal{A}$. Since $\mathcal{A}$ is a subalgebra of $\hrts$ containing all $\bullet_x$ and closed under $B_{\omega}^{+}$ for all $\omega \in \Omega$, $\mathcal{A}=\hrts$ and hence $\coll$ is coassociative. Next, define
\begin{align*}
\mathcal{B}:=\{F \in \hrts \, \mid \, (\counit \ot \id)\coll(F)=\beta_{l}(F) \, \text{ and } \, (\id \ot \counit)\coll(F)=\beta_{r}(F)  \},
\end{align*}
where $\beta_{l}: \hrts \rightarrow \bfk \ot \hrts, F \mapsto 1_{\bfk} \ot F$ and $\beta_{r}: \hrts \rightarrow \hrts \ot \bfk, F \mapsto F \ot 1_{\bfk}$. Note that $1 \in \mathcal{B}$ by $\coll(1)=1 \ot 1$ and $\counit(1)=1_{\bfk}$. By Lemmas~\mref{lem:colmor} and \mref{lem:counitmor}, $\coll$ and $\counit$ are algebra homomorphisms of $\hrts$, so $\mathcal{B}$ is a subalgebra of $\hrts$. For any $x \in X$, we have
\begin{align*}
(\counit \ot \id)\coll(\bullet_{x})=(\counit \ot \id)(\bullet_{x} \ot \etree + \etree \ot \bullet_{x}+ \lambda \bullet_{x} \ot \bullet_{x})=1_{\bfk} \ot \bullet_{x}
\end{align*}
and
\begin{align*}
(\id \ot \counit)\coll(\bullet_{x})=(\id \ot \counit)(\bullet_{x} \ot \etree + \etree \ot \bullet_{x} + \lambda \bullet_{x} \ot \bullet_{x})=\bullet_{x} \ot \etree,
\end{align*}
hence $\bullet_{x} \in \mathcal{B}$. For $F \in \mathcal{B}$ and $\omega \in \Omega$, we have
\begin{align*}
(\counit \ot \id)\coll(B_{\omega}^{+}(F))=(\counit \ot \id)(B_{\omega}^{+} \ot \id + \id \ot B_{\omega}^{+})\coll(F)=(\id \ot B_{\omega}^{+})(\counit \ot \id)\coll(F)
\end{align*}
and
\begin{align*}
(\id \ot \counit)\coll(B_{\omega}^{+}(F))=(\id \ot \counit)(B_{\omega}^{+} \ot \id + \id \ot B_{\omega}^{+})\coll(F)=(B_{\omega}^{+} \ot \id)(\id \ot \counit)\coll(F).
\end{align*}
As $F \in \mathcal{B}$, $(\counit \ot \id) \coll(F)=\beta_{l}(F)$ and $(\id \ot \counit)\coll(F)=\beta_{r}(F)$ and so
\begin{align*}
(\counit \ot \id)\coll(B_{\omega}^{+}(F))=(\id \ot B_{\omega}^{+}) \beta_{l}(F)=\beta_{l}(B_{\omega}^{+}(F))
\end{align*}
and
\begin{align*}
(\id \ot \counit)\coll(B_{\omega}^{+}(F))=(B_{\omega}^{+} \ot \id)\beta_{r}(F)=\beta_{r}(B_{\omega}^{+}(F)),
\end{align*}
which means $B_{\omega}^{+}(F) \in \mathcal{B}$. Since $\mathcal{B}$ is a subalgebra of $\hrts$ containing all $\bullet_x$ and closed under $B_{\omega}^{+}$ for all $\omega \in \Omega$, we have $\mathcal{B}=\hrts$ and so $\counit$ is a counit. Thus the triple $(\hrts, \,\coll, \, \varepsilon)$ is a coalgebra.
\end{proof}

Now we arrive at our main result in this subsection.

\begin{theorem} \mlabel{thm:main}
The quintuple $(\hrts, \conc, \etree, \coll, \counit)$ is a bialgebra.
\end{theorem}

\begin{proof}
Note that $(\hrts, \conc, \etree)$ is a unitary algebra. Then the result follows from Lemmas~\mref{lem:colmor}, \mref{lem:cclosed}, \mref{lem:counitmor} and Theorem~\mref{thm:coalgebra}.
\end{proof}
\subsection{A combinatorial description of $\coll$}
In this subsection, we give a combinatorial description of $\coll$. For simplicity, we denote the coproduct $\col_0$ by $\col$.

First we recall some notations for trees. As a rooted tree is a partially ordered set, for each vertex $v$ of a rooted tree, we call the vertices greater than $v$ the {\bf ancestors} of $v$ and the vertices lower than $v$ the {\bf descendants} of $v$. In particular, if $u$ covers $v$ in the partial order, we call $u$ the {\bf parent} of $v$ and $v$ a {\bf child} of $u$.

For a decorated planar rooted forest $F$, denote by $V(F)$ the set of vertices of $F$. Take a subset $I \subseteq V(F)$ and define the {\bf induced subforest} $F_I$ of $F$ to be the decorated planar rooted forests as follows:
\begin{enumerate}
\item the vertices of $F_I$ are the elements of $I$;

\item for each $v \in I$, if there are ancestors of $v$ in $I$, draw an incoming edge to $v$ from the smallest ancestor belonging to I (the smallest ancestor of $v$ exists since the ancestors of $v$ form a chain). If there is no ancestors of $v$ in $I$, then $v$ is a root in $F_I$.
\item since $F$ is a planar rooted forest, there is a unique way to arrange the trees in $F_I$ so that $F_I$ is a decorated planar rooted forest.
\end{enumerate}

\begin{remark}
\begin{enumerate}
\item For a leaf vertex $v \in I$, we see that $v$ is still a leaf in $F_I$ and so $F_I \in \hrts$ for all $I \subseteq V(F)$.

\item Since there are $2^{|V(F)|}$ subsets of $T_0$, there are $2^{|V(F)|}$ induced forests of $T$.

\item If $I=\emptyset$, then $F_{\emptyset}=1$; if $I=V(F)$, then $F_{V(F)}=F$.
\end{enumerate}
\end{remark}

\begin{exam}
Let $F=\tdun{$\alpha$} \tdtroisun{$\beta$}{$y$}{$x$}$.
Since the decorations of $F$ are all different, we use the decorations of the vertices to denote the vertices here. Then, all induced subforests of $F$ are:

\begin{table}[!htbp]
\centering
\begin{tabular}{|c|c|c|c|c|c|c|c|c|c|c|c|c|c|c} 
\hline
\text{subsets of $V(F)$}& $\emptyset$ & $\{ \alpha \}$ & $\{ \beta \}$ & $\{ x \}$ & $\{ y\}$ & $\{ \alpha,\beta \}$ & $\{\alpha,x \}$ & $\{\alpha,y \}$\\
\hline
\text{induced subforests} & $\etree$ & $\tdun{$\alpha$}$ & $\tdun{$\beta$}$ & $\tdun{$x$}$ & $\tdun{$y$}$& $\tdun{$\alpha$} \tdun{$\beta$}$ & $\tdun{$\alpha$} \tdun{$x$}$ & $\tdun{$\alpha$} \tdun{$y$}$ \\
\hline
\text{subsets of $V(F)$}& $\{\beta,x \}$ & $\{\beta,y \}$ & $\{x,y\}$ & $\{\alpha,\beta,x \}$ & $\{\alpha,\beta,y \}$ & $\{\alpha,x,y\}$ & $\{\beta,x,y \}$ & $\{\alpha,\beta,x,y\}$\\
\hline
\text{induced subforests} & $\tddeux{$\beta$}{$x$}$ & $\tddeux{$\beta$}{$y$}$ & $\tdun{$x$}$\tdun{$y$} & $\tdun{$\alpha$}\tddeux{$\beta$}{$x$}$ & $\tdun{$\alpha$}\tddeux{$\beta$}{$y$}$ & $\tdun{$\alpha$}\tdun{$x$}\tdun{$y$}$ & $\tdtroisun{$\beta$}{$y$}{$x$}$ & $\tdun{$\alpha$}\tdtroisun{$\beta$}{$y$}{$x$}$\\
\hline
\end{tabular} 
\end{table}
\end{exam}

Now we give a combinatorial description of the coproduct $\coll$.

\begin{theorem}\mlabel{thm:comcoproduct}
	For any decorated planar rooted forest $F \in \rfs$,
	\begin{align}
		\Delta_\lambda(F)=\sum \limits_{\substack{V(F)=I\cup J,\\ I\cap J\subseteq V_X(F)}}\lambda^{|I\cap J|} F_{I} \ot F_{J},
		\mlabel{eq:comcoproduct1}
	\end{align}
where $V_X(F)$ is the set of vertices of $F$ which are decorated by an element of $X$.
In particular, for $\lambda=0$,
\begin{align}
	\Delta(F)=\sum \limits_{I \subseteq V(F)} F_{I} \ot F_{V(F)\backslash I}.
	\mlabel{eq:comcoproduct}
\end{align}
\end{theorem}

\begin{proof}
	Denote the righthand of Eq.~(\mref{eq:comcoproduct1}) by $C_\lambda(F)$. First we show that $C_\lambda$ is an algebra homomorphism. Suppose $F=F_1F_2$, where $F_1,F_2$ are non-trivial decorated planar rooted forests. Then
	
\begin{align*}
C_\lambda(F_1)C_\lambda(F_2)&=\sum_{\substack{V(F_1)=I_1\cup J_1,\\V(F_2)=I_2\cup J_2,\\ I_1\cap J_1\subseteq V_X(F_1),\\ I_2\cap J_2\subseteq V_X(F_2)}}
\lambda^{|I_1\cap J_1|+|I_2\cap J_2|} {F_1}_{I_1}{F_2}_{I_2}\otimes {F_1}_{J_1}{F_2}_{J_2}\\
&=\sum_{\substack{V(F_1)=I_1\cup J_1,\\V(F_2)=I_2\cup J_2,\\ I_1\cap J_1\subseteq V_X(F_1),\\ I_2\cap J_2\subseteq V_X(F_2)}}
\lambda^{|(I_1\cup I_2)\cap (J_1\cup J_2)|}
(F_1F_2)_{I_1\cup I_2}\otimes (F_1F_2)_{J_1\cup J_2}\\
&=\sum_{\substack{V(F)=I\cup J,\\ I\cap J\subseteq V_X(F_1F_2)}}
(F_1F_2)_I\otimes (F_1F_2)_J\\
&=C_\lambda(F_1F_2),
\end{align*}
where $I=I_1\cup I_2$ and $J=J_1\cup J_2$, as $V_X(F_1F_2)=V_X(F_1)\cup V_X(F_2)$.	So $C_\lambda$ is an algebra morphism.
Moreover, $C_{\lambda} (\etree)= \etree \ot \etree$ and for $x\in X$,
\[C_\lambda(\tdun{$x$})=\tdun{$x$}\otimes 1+1\otimes \tdun{$x$}+\lambda \tdun{$x$}\otimes \tdun{$x$}.\]
Let $G\in \rfs$ and $\omega \in \Omega$. We put $F=B_\omega^+(G)$
and we denote by $r$ the root of $F$. Note that $r\notin V_X(F)$, and that $V_X(F)=V_X(G)$. Therefore,
\begin{align*}
C_\lambda(F)&=\sum_{\substack{V(F)=I\cup J,\\ I\cap J\subseteq V_X(G),\\ r\in I}}\lambda^{|I\cap J|}B_\omega^+(G)_I\otimes G_J
+\sum_{\substack{V(F)=I\cup J,\\ I\cap J\subseteq V_X(G),\\ r\in J}}\lambda^{|I\cap J|}G_I\otimes B_\omega^+(G)_J\\
&=\sum_{\substack{V(G)=I\cup J,\\ I\cap J\subseteq V_X(G)}}\lambda^{|I\cap J|}B_\omega^+(G_I)\otimes G_J
+\sum_{\substack{V(G)=I\cup J,\\ I\cap J\subseteq V_X(G),\\ r\in J}}\lambda^{|I\cap J|}G_I\otimes B_\omega^+(G_J)\\
&=(\id \otimes B_\omega^++B_\omega^+\otimes \id) C_\lambda(G),
\end{align*}
so $C_\lambda$ satisfies (\ref{eq:dele}), (\ref{eq:dbp}), (\ref{eq:cdbp}) and (\ref{eq:delee1}): this implies that $C_\lambda=\Delta_\lambda$.
\end{proof}

\begin{exam}
By Theorem~\mref{thm:comcoproduct}, we have
\begin{align*}
C_{\lambda}(\tdun{$x$}\tddeux{$\alpha$}{$y$})=&\ \tdun{$x$}\tddeux{$\alpha$}{$y$} \ot \etree+ \tddeux{$\alpha$}{$y$} \ot \tdun{$x$}
+\tdun{$x$}\tdun{$\alpha$} \ot \tdun{$y$}+\tdun{$\alpha$} \ot \tdun{$x$}\tdun{$y$} +\tdun{$x$}\tdun{$y$}\ot \tdun{$\alpha$}+ \tdun{$y$} \ot \tdun{$x$}\tdun{$\alpha$}\\
&\ +\tdun{$x$} \ot \tddeux{$\alpha$}{$y$}+ \etree \ot \tdun{$x$}\tddeux{$\alpha$}{$y$}+ \lambda \tdun{$x$}\tddeux{$\alpha$}{$y$} \ot \tdun{$x$}+ \lambda \tdun{$x$}\tdun{$\alpha$} \ot \tdun{$x$}\tdun{$y$}+\lambda \tddeux{$\alpha$}{$y$} \ot \tdun{$x$} \tdun{$y$}+\lambda \tdun{$x$}\tddeux{$\alpha$}{$y$} \ot \tdun{$y$}\\
&\ +\lambda^2 \tdun{$x$}\tddeux{$\alpha$}{$y$} \ot \tdun{$x$}\tdun{$y$}+ \lambda \tdun{$x$} \tdun{$y$} \ot \tdun{$x$} \tdun{$\alpha$}
+\lambda \tdun{$x$} \ot \tdun{$x$}\tddeux{$\alpha$}{$y$}+ \lambda \tdun{$x$} \tdun{$y$} \ot \tddeux{$\alpha$}{$y$}+ \lambda \tdun{$y$} \ot \tdun{$x$}\tddeux{$\alpha$}{$y$}+\lambda^2 \tdun{$x$}\tdun{$y$} \ot \tdun{$x$}\tddeux{$\alpha$}{$y$} ,\\
C_{\lambda}( \tdtroisun{$\alpha$}{$\beta$}{$x$})=&\ \tdtroisun{$\alpha$}{$\beta$}{$x$} \ot \etree+ \tdun{$x$}\tdun{$\beta$} \ot \tdun{$\alpha$}+\tddeux{$\alpha$}{$\beta$} \ot \tdun{$x$}+\tdun{$\beta$} \ot \tddeux{$\alpha$}{$x$}+\tddeux{$\alpha$}{$x$} \ot \tdun{$\beta$}+\tdun{$x$} \ot \tddeux{$\alpha$}{$\beta$}+ \tdun{$\alpha$} \ot \tdun{$x$}\tdun{$\beta$}+ \etree \ot \tdtroisun{$\alpha$}{$\beta$}{$x$}\\
&\ +\lambda \tdtroisun{$\alpha$}{$\beta$}{$x$} \ot \tdun{$x$}+ \lambda \tdun{$x$} \tdun{$\beta$} \ot \tddeux{$\alpha$}{$x$}+ \lambda \tddeux{$\alpha$}{$x$} \ot \tdun{$x$}\tdun{$\beta$}+ \lambda \tdun{$x$} \ot \tdtroisun{$\alpha$}{$\beta$}{$x$},
\end{align*}
which are consistent with $\coll(\tdun{$x$}\tddeux{$\alpha$}{$y$})$ and $\coll( \tdtroisun{$\alpha$}{$\beta$}{$x$})$ in Example~\mref{exam:cop}.
\end{exam}

\subsection{Dual of the coproduct $\Delta$}

We identify $\hrts$ with its graded dual through the symmetric bilinear form defined by
\begin{align*}
&\forall F,G\in \rfs, \quad \langle F,G\rangle=\delta_{F,G}.
\end{align*}
As the coproduct $\Delta=\Delta_0$ is homogeneous, it induces a product $\star$ on $\hrts$, such that
\begin{align*}
&\forall x,y,z\in \hrts,&\langle x\otimes y,\Delta(z)\rangle&=\langle x\star y,z\rangle.
\end{align*}
For any $F,G,H\in \rfs$, we put
\[N(F,G;H)=\{I\subseteq V(H)\mid H_{I}=F,\: H_{ V(H)\setminus I}=G\}.\]
Then, by Theorem \ref{thm:comcoproduct}, for any $H\in \rfs$,
\[\Delta(H)= \sum_{F,G\in \rfs} |N(F,G;H)| F\otimes G.\]
Therefore, for any $F,G\in \rfs$,
\[F\star G=\sum_{H\in \rfs} |N(F,G;H)| H.\]
In order to describe these sets $N(F,G;H)$, we introduce an alternative description of planar rooted forests. If $F\in \rfs$, recall that its vertices are given two partial orders $\leq_r$ and $\leq_h$: if $x,y\in V(F)$, $x\leq_h y$ if there exists an oriented path from $x$ to $y$ in $F$, the edges of $F$ being oriented from the roots to the leaves; $x\leq_r y$ if $x$ and $y$ are not comparable for $\leq_h$ and if in the plane representation of $F$,
$y$ is more on the right than $x$. Then:
\begin{itemize}
\item For any $x,y\in V(F)$, if $x$ and $y$ are comparable for both $\leq_h$ and $\leq_r$, then $x=y$.
\item $\leq_{h,r}=\leq_h\cup \leq_r$ is a total order on $V(F)$.
\end{itemize}
Therefore, if $|V(F)|=n$, its vertices are ordered by $\leq_{h,r}$:
\[v_1\leq_{h,r}\cdots \leq_{h,r} v_n.\]
For example, we represent this total order on forests with 4 vertices by indices:
\begin{align*}
&\tdun{$1$}\tdun{$2$}\tdun{$3$}\tdun{$4$},&&\tdun{$1$}\tdun{$2$}\tddeux{$3$}{$4$},&&\tdun{$1$}\tddeux{$2$}{$3$}\tdun{$4$},&&\tddeux{$1$}{$2$}\tdun{$3$}\tdun{$4$},&&\tddeux{$1$}{$2$}\tddeux{$3$}{$4$},&&\tdun{$1$}\tdtroisun{$2$}{$4$}{$3$},&&\tdtroisun{$1$}{$3$}{$2$}\tdun{$4$},\\
&\tdun{$1$}\tdtroisdeux{$2$}{$3$}{$4$}&&\tdtroisdeux{$1$}{$2$}{$3$}\tdun{$4$},&&\tdquatreun{$1$}{$4$}{$3$}{$2$},&&\tdquatredeux{$1$}{$4$}{$2$}{$3$},&&\tdquatretrois{$1$}{$3$}{$4$}{$2$},&&\tdquatrequatre{$1$}{$2$}{$4$}{$3$},&&\tdquatrecinq{$1$}{$2$}{$3$}{$4$}.
\end{align*}

We then represent $F\in \rfs$ by a matrix $M(F)$ of size $n\times(n+1)$. The rows are indexed from $1$ to $n$, and the columns are indexed from $0$ to $n$. Then:
\begin{enumerate}
\item If $1\leq i\leq n$, $M(F)_{i,0}$ is the decoration of the vertex $v_i$.
\item If $i>j$, $M(F)_{i,j}=0$.
\item If $i=j$, $M(F)_{i,j}$ is the symbol $=$.
\item If $i<j$ and $v_i\leq_h v_j$, then $M(F)_{i,j}$ is the symbol $h$.
\item  If $i<j$ and $v_i\leq_r v_j$, then $M(F)_{i,j}$ is the symbol $r$.
\end{enumerate}
Here are examples of matrices $M(F)$: here, $x,y\in X$, $\alpha,\beta,\gamma \in \Omega$.
\begin{align*}
M(\tdun{$x$})&=\begin{pmatrix}
x&=
\end{pmatrix},\\
M(\tddeux{$\alpha$}{$x$})&=\begin{pmatrix}
\alpha&=&h\\
x&0&=
\end{pmatrix},&
M(\tdun{$\alpha$}\tdun{$x$})&=\begin{pmatrix}
\alpha&=&r\\
x&0&=
\end{pmatrix},\\
M(\tdtroisun{$\alpha$}{$y$}{$x$})&=\begin{pmatrix}
\alpha&=&h&h\\
x&0&=&r\\
y&0&0&=
\end{pmatrix},&
M(\tdtroisdeux{$\alpha$}{$\beta$}{$x$})&=\begin{pmatrix}
\alpha&=&h&h\\
\beta&0&=&h\\
x&0&0&=
\end{pmatrix}.
\end{align*}
This leads to the following definition:

\begin{defn}
Let $A$ be a matrix with $n$ rows and $n+1$ columns. We shall say that it is {\bf forest-representable} if:
\begin{enumerate}
\item If $1\leq i\leq n$, $A_{i,0}\in X\cup \Omega$.
\item If $1\leq j<i \leq n$, $A_{i,j}=0$.
\item If $1\leq i\leq n$, $A_{i,i}$ is the symbol $=$.
\item If $1\leq i<j \leq n$, $A_{i,j}\in \{h,r\}$.
\item If $1\leq i\leq n$ and $A_{i,0}\in X$, then for any $j>i$, $A_{i,j}=r$.
\item If $1\leq i<j<k\leq n$, then
\begin{align*}
\begin{pmatrix}
A_{i,j}&A_{i,k}\\
A_{j,j}&A_{j,k}
\end{pmatrix}\in
\left\{\begin{pmatrix}
r&r\\
=&r
\end{pmatrix},
\begin{pmatrix}
h&h\\
=&h
\end{pmatrix},\begin{pmatrix}
r&r\\
=&h
\end{pmatrix},\begin{pmatrix}
h&h\\
=&r
\end{pmatrix},\begin{pmatrix}
h&r\\
=&r
\end{pmatrix}
\right\}.
\end{align*}
\end{enumerate}
The set of forest-representable matrices is denoted by $\mathcal{FM}(X,\Omega)$.
\end{defn}

\begin{prop}
The correspondence $F\longrightarrow M(F)$ is a bijection from $\rfs$ to $\mathcal{FM}(X,\Omega)$.
\end{prop}

\begin{proof}
Let $F\in \rfs$, let us prove that $M(F)\in \mathcal{FM}(X,\Omega)$.
Points (a)-(d) are obvious. If $M(F)_{i,0}\in X$, then necessarily $v_i$ is a leaf, so if $j>i$, $v_i\leq_h v_j$ is not possible.
Therefore, $v_i\leq_r v_j$, and $M(F)_{i,j}=r$. This proves (e).
Let us now prove (f). Let $1\leq i<j<k\leq n$.
If $v_i\leq_h v_j$ and $v_j\leq_h v_j$, then, as $\leq_h$ is an order, $v_i\leq_h v_k$. Consequently,
\[\begin{pmatrix}
	M(F)_{i,j}&M(F)_{i,k}\\
	M(F)_{j,j}&m_{j,k}
\end{pmatrix}\neq \begin{pmatrix}
	h&r\\
	=&h
\end{pmatrix}.\]
Similarly, as $\leq_r$ is an order,
\[\begin{pmatrix}
	M(F)_{i,j}&M_(F)_{i,k}\\
	M(F)_{j,j}&M(F)_{j,k}
\end{pmatrix}\neq \begin{pmatrix}
	r&h\\
	=&r
\end{pmatrix}.\]
If $v_i\neq_r v_j$ and $v_j \leq_h v_k$, as $F$ is a plane forest, then necessarily $v_i \leq_r v_k$, therefore,
\[\begin{pmatrix}
	M(F)_{i,j}&M(F)_{i,k}\\
	M(F)_{j,j}&M(F)_{j,k}
\end{pmatrix}\neq \begin{pmatrix}
	r&h\\
	=&h
\end{pmatrix}.\]
This finally proves (f). Therefore, the map $M:\rfs \longrightarrow \mathcal{MF}(X,\Omega)$ is well-defined.
Note that $M(F)$ allows to obtain $\leq_h$ and $\leq_r$, after identification of $V(F)$ with $\{1,\ldots,n\}$. So $M$ is injective.
Let $A\in \mathcal{FM}(X,\Omega)$. We denote by $n$ the number of rows of $A$. We define two relations $\leq_h$ and $\leq_r$ on $\{1,\ldots,n\}$ by
\begin{align*}
&\forall i,j\in \{1,\ldots,n\},& i\leq_h j\mbox{ if }A_{i,j}\in \{=,h\},\\
&&i\leq_r j\mbox{ if }A_{i,j}\in \{=,r\}.
\end{align*}
If $1\leq i<j<k\leq n$,
\[\begin{pmatrix}
	A_{i,j}&A_{i,k}\\
	A_{j,j}&A_{j,k}
\end{pmatrix}\neq \begin{pmatrix}
	h&r\\
	=&h
\end{pmatrix},\]
so if $A_{i,j}=h$ and $A_{j,k}=h$, then $A_{i,k}=h$,
so $\leq_h$ is an order. Similarly, $\leq_r$ is an order. Let $F$ be the Hasse graph of $\leq_h$. If $1\leq i<j<k\leq n$,
\[\begin{pmatrix}
	A_{i,j}&A_{i,k}\\
	A_{j,j}&A_{j,k}
\end{pmatrix}\neq \begin{pmatrix}
	r&h\\
	=&h
\end{pmatrix},\]
so $F$ is a rooted forest. The relation $\leq_r$ induces a plane structure on $F$, so $F$ is a planar forest. We decorate the vertex $i$ of $F$ by $A_{i,0}$.
By condition (e), if $i$ is decorated by $x\in X$, then $i$ is a leaf of $F$, so $F\in \rfs$. By construction, $M(F)=A$. So $A$ is surjective.
\end{proof}

\begin{defn}
Let $k,l\in \mathbb{N}$.
\begin{enumerate}
\item Let $\sigma \in \mathfrak{S}_{k+l}$ be a permutation. Then $\sigma$ is a {\bf $(k,l)$-shuffle} if
\begin{align*}
&\sigma(1)<\cdots<\sigma(k),&&\sigma(k+1)<\cdots<\sigma(k+l).
\end{align*}
The set of $(k,l)$-shuffles is denoted by $\sh(k,l)$.
\item Let $A,B\in \mathcal{FM}(X,\Omega)$, with respectively $k$ and $l$ columns, and $\sigma\in \sh(k,l)$. We denote by $\mathcal{FM}_\sigma(A,B)$ the set of forest-representable matrices $C$ with $k+l$ rows such that:
\begin{itemize}
\item If $1\leq i\leq k$, $C_{\sigma(i),0}=A_{i,0}$.
\item If $k+1\leq i\leq k+l$, $C_{\sigma(i),0}=B_{i-k,0}$.
\item If $1\leq i,j\leq k$, $C_{\sigma(i),\sigma(j)}=A_{i,j}$.
\item If $k+1\leq i,j\leq k+l$, $C_{\sigma(i),\sigma(j)}=B_{i-k,j-k}$.
\end{itemize}
\end{enumerate}
\end{defn}

\begin{exam}\label{examproduct1}
Let $\alpha,\beta,\gamma \in \Omega$.
We consider
\begin{align*}
F&=\tdun{$\alpha$}, &G&=\tddeux{$\beta$}{$\gamma$},\\
A=M(F)&=\begin{pmatrix}
\alpha&=
\end{pmatrix},&B=M(G)&=\begin{pmatrix}
\beta&=&h\\
\gamma&0&=
\end{pmatrix}.
\end{align*}
There are three $(1,2)$-shuffles.
\begin{enumerate}
\item For $\sigma=(123)$, we look for forest-representable matrices of the form
\[\begin{pmatrix}
\alpha&=&\mbox{$h$ or $r$}&\mbox{$h$ or $r$}\\
\beta&0&=&h\\
\gamma&0&0&=
\end{pmatrix}.\]
There are two of them:
\[\mathcal{FM}_\sigma(A,B)
=\left\{
\begin{pmatrix}
	\alpha&=&h&h\\
	\beta&0&=&h\\
	\gamma&0&0&=
\end{pmatrix},
\begin{pmatrix}
	\alpha&=&r&h\\
	\beta&0&=&h\\
	\gamma&0&0&=
\end{pmatrix}
\right\}=\left\{M(\tdtroisdeux{$\alpha$}{$\beta$}{$\gamma$}),M(\tdun{$\alpha$}\tddeux{$\beta$}{$\gamma$})\right\}.\]
\item For $\sigma=(213)$, we look for forest-representable matrices of the form
\[\begin{pmatrix}
	\beta&=&\mbox{$h$ or $r$}&h\\
	\alpha&0&=&\mbox{$h$ or $r$}\\
	\gamma&0&0&=
\end{pmatrix}.\]
There are two of them:
\[\mathcal{FM}_\sigma(A,B)
=\left\{
\begin{pmatrix}
	\beta&=&h&h\\
	\alpha&0&=&h\\
	\gamma&0&0&=
\end{pmatrix},
\begin{pmatrix}
	\beta&=&h&h\\
	\alpha&0&=&r\\
	\gamma&0&0&=
\end{pmatrix}
\right\}=\left\{M(\tdtroisdeux{$\beta$}{$\alpha$}{$\gamma$}),M(\tdtroisun{$\beta$}{$\gamma$}{$\alpha$})\right\}.\]
\item For $\sigma=(312)$, we look for forest-representable matrices of the form
\[\begin{pmatrix}
	\beta&=&h&\mbox{$h$ or $r$}\\
	\gamma&0&=&\mbox{$h$ or $r$}\\
	\alpha&0&0&=
\end{pmatrix}.\]
There are three of them:
\[\mathcal{FM}_\sigma(A,B)
=\left\{
\begin{pmatrix}
	\beta&=&h&h\\
	\gamma&0&=&h\\
	\alpha&0&0&=
\end{pmatrix},\begin{pmatrix}
\beta&=&h&h\\
\gamma&0&=&r\\
\alpha&0&0&=
\end{pmatrix},
\begin{pmatrix}
	\beta&=&h&r\\
	\gamma&0&=&r\\
	\alpha&0&0&=
\end{pmatrix}
\right\}=\left\{M(\tdtroisdeux{$\beta$}{$\gamma$}{$\alpha$}),M(\tdtroisun{$\beta$}{$\alpha$}{$\gamma$}),
M(\tddeux{$\beta$}{$\gamma$}\tdun{$\alpha$})\right\}.\]
\end{enumerate}
\end{exam}

\begin{exam}\label{examproduct2}
	Let $x\in X$ and $\beta,\gamma \in \Omega$.
	We consider
	\begin{align*}
		F&=\tdun{$x$}, &G&=\tddeux{$\beta$}{$\gamma$},\\
		A=M(F)&=\begin{pmatrix}
			x&=
		\end{pmatrix},&B=M(G)&=\begin{pmatrix}
			\beta&=&h\\
			\gamma&0&=
		\end{pmatrix}.
	\end{align*}
	There are three $(1,2)$-shuffles.
	\begin{enumerate}
		\item For $\sigma=(123)$, we look for forest-representable matrices of the form
		\[\begin{pmatrix}
			x&=&\mbox{$h$ or $r$}&\mbox{$h$ or $r$}\\
			\beta&0&=&h\\
			\gamma&0&0&=
		\end{pmatrix}.\]
		There is only one:
		\[\mathcal{FM}_\sigma(A,B)
		=\left\{
			\begin{pmatrix}
			x&=&r&r\\
			\beta&0&=&h\\
			\gamma&0&0&=
		\end{pmatrix}
		\right\}=\left\{M(\tdun{$x$}\tddeux{$\beta$}{$\gamma$})\right\}.\]
		\item For $\sigma=(213)$, we look for forest-representable matrices of the form
		\[\begin{pmatrix}
			\beta&=&\mbox{$h$ or $r$}&h\\
			x&0&=&\mbox{$h$ or $r$}\\
			\gamma&0&0&=
		\end{pmatrix}.\]
		There is only one:
		\[\mathcal{FM}_\sigma(A,B)
		=\left\{
		\begin{pmatrix}
			\beta&=&h&h\\
			x&0&=&r\\
			\gamma&0&0&=
		\end{pmatrix}
		\right\}=\left\{M(\tdtroisun{$\beta$}{$\gamma$}{$x$})\right\}.\]
		\item For $\sigma=(311)$, we look for forest-representable matrices of the form
		\[\begin{pmatrix}
			\beta&=&h&\mbox{$h$ or $r$}\\
			\gamma&0&=&\mbox{$h$ or $r$}\\
			x&0&0&=
		\end{pmatrix}.\]
		There are three of them:
		\[\mathcal{FM}_\sigma(A,B)
		=\left\{
		\begin{pmatrix}
			\beta&=&h&h\\
			\gamma&0&=&h\\
			x&0&0&=
		\end{pmatrix},\begin{pmatrix}
			\beta&=&h&h\\
			\gamma&0&=&r\\
			x&0&0&=
		\end{pmatrix},
		\begin{pmatrix}
			\beta&=&h&r\\
			\gamma&0&=&r\\
			x&0&0&=
		\end{pmatrix}
		\right\}=\left\{M(\tdtroisdeux{$\beta$}{$\gamma$}{$x$}),M(\tdtroisun{$\beta$}{$x$}{$\gamma$}),
		M(\tddeux{$\beta$}{$\gamma$}\tdun{$x$})\right\}.\]
	\end{enumerate}
\end{exam}

\begin{exam}\label{examproduct3}
Let $x\in X$ and $\alpha,\beta,\gamma \in \Omega$. We consider
\begin{align*}
F&=\tdun{$\alpha$}\tdun{$x$},&G&=\tddeux{$\beta$}{$\gamma$},\\
A=M(F)&=\begin{pmatrix}
\alpha&=&h\\
x&0&=
\end{pmatrix},&
B=M(G)&=\begin{pmatrix}
\beta&=&h\\
\gamma&0&=
\end{pmatrix}.
\end{align*}
There are six $(2,2)$-shuffles:
\begin{align*}
\mathcal{FM}_{(1234)}(A,B)&=\left\{
\begin{pmatrix}
\alpha&=&r&\textcolor{red}{r}&\textcolor{red}{r}\\
x&0&=&\textcolor{red}{r}&\textcolor{red}{r}\\
\beta&0&0&=&h\\
\gamma&0&0&0&=
\end{pmatrix}
\right\}=\left\{M(\tdun{$\alpha$}\tdun{$x$}\tddeux{$\beta$}{$\gamma$})\right\},\\
\mathcal{FM}_{(1324)}(A,B)&=\left\{
\begin{pmatrix}
\alpha&=&\textcolor{red}{r}&r&\textcolor{red}{r}\\
\beta&0&=&\textcolor{red}{h}&h\\
x&0&0&=&\textcolor{red}{r}\\
\gamma&0&0&0&=
\end{pmatrix}
\right\}=\left\{M(\tdun{$\alpha$} \tdtroisun{$\beta$}{$\gamma$}{$x$})\right\},\\
\mathcal{FM}_{(1423)}(A,B)&=\left\{
\begin{pmatrix}
\alpha&=&\textcolor{red}{h}&\textcolor{red}{h}&r\\
\beta&0&=&h&\textcolor{red}{r}\\
\gamma&0&0&=&\textcolor{red}{r}\\
x&0&0&0&=
\end{pmatrix},\begin{pmatrix}
\alpha&=&\textcolor{red}{r}&\textcolor{red}{r}&r\\
\beta&0&=&h&\textcolor{red}{r}\\
\gamma&0&0&=&\textcolor{red}{r}\\
x&0&0&0&=
\end{pmatrix},\begin{pmatrix}
\alpha&=&\textcolor{red}{r}&\textcolor{red}{r}&r\\
\beta&0&=&h&\textcolor{red}{h}\\
\gamma&0&0&=&\textcolor{red}{r}\\
x&0&0&0&=
\end{pmatrix},\begin{pmatrix}
\alpha&=&\textcolor{red}{r}&\textcolor{red}{r}&r\\
\beta&0&=&h&\textcolor{red}{h}\\
\gamma&0&0&=&\textcolor{red}{h}\\
x&0&0&0&=
\end{pmatrix}
\right\}\\
&=\left\{M(\tdtroisdeux{$\alpha$}{$\beta$}{$\gamma$}\tdun{$x$}),M(\tdun{$\alpha$}\tddeux{$\beta$}{$\gamma$}\tdun{$x$}),M(\tdun{$\alpha$}\tdtroisun{$\beta$}{$x$}{$\gamma$}),M(\tdun{$\alpha$} \tdtroisdeux{$\beta$}{$\gamma$}{$x$})\right\},\\
\mathcal{FM}_{(2314)}(A,B)&=\left\{
\begin{pmatrix}
\beta&=&\textcolor{red}{h}&\textcolor{red}{h}&h\\
\alpha&0&=&r&\textcolor{red}{r}\\
x&0&0&=&\textcolor{red}{r}\\
\gamma&0&0&0&=
\end{pmatrix}
\right\}=\left\{M(\tdquatreun{$\beta$}{$\gamma$}{$x$}{$\alpha$})\right\},\\
\mathcal{FM}_{(2413)}(A,B)&=\left\{
\begin{pmatrix}
\beta&=&\textcolor{red}h{}&h&\textcolor{red}{r}\\
\alpha&0&=&\textcolor{red}{h}&r\\
\gamma&0&0&=&\textcolor{red}{r}\\
x&0&0&0&=
\end{pmatrix},\begin{pmatrix}
\beta&=&\textcolor{red}{h}&h&\textcolor{red}{h}\\
\alpha&0&=&\textcolor{red}{h}&r\\
\gamma&0&0&=&\textcolor{red}{r}\\
x&0&0&0&=
\end{pmatrix},\begin{pmatrix}
\beta&=&\textcolor{red}{h}&h&\textcolor{red}{r}\\
\alpha&0&=&\textcolor{red}{r}&r\\
\gamma&0&0&=&\textcolor{red}{r}\\
x&0&0&0&=
\end{pmatrix},\begin{pmatrix}
\beta&=&\textcolor{red}{h}&h&\textcolor{red}{h}\\
\alpha&0&=&\textcolor{red}{r}&r\\
\gamma&0&0&=&\textcolor{red}{r}\\
x&0&0&0&=
\end{pmatrix}
\right\}\\
&=\left\{
M(\tdtroisdeux{$\beta$}{$\alpha$}{$\gamma$}\tdun{$x$}),
M(\tdquatretrois{$\beta$}{$\gamma$}{$x$}{$\alpha$}),
M(\tdquatredeux{$\beta$}{$x$}{$\alpha$}{$\gamma$}),
M(\tdquatreun{$\beta$}{$x$}{$\gamma$}{$\alpha$})
\right\},\\
\mathcal{FM}_{()}(A,B)&=\left\{
\begin{array}{c}
\begin{pmatrix}
\beta&=&h&\textcolor{red}{h}&\textcolor{red}{h}\\
\gamma&0&=&\textcolor{red}{h}&\textcolor{red}{h}\\
\alpha&0&0&=&r\\
x&0&0&0&=
\end{pmatrix},\begin{pmatrix}
\beta&=&h&\textcolor{red}{h}&\textcolor{red}{h}\\
\gamma&0&=&\textcolor{red}{r}&\textcolor{red}{r}\\
\alpha&0&0&=&r\\
x&0&0&0&=
\end{pmatrix},\begin{pmatrix}
\beta&=&h&\textcolor{red}{h}&\textcolor{red}{r}\\
\gamma&0&=&\textcolor{red}{r}&\textcolor{red}{r}\\
\alpha&0&0&=&r\\
x&0&0&0&=
\end{pmatrix},\\
\begin{pmatrix}
	\beta&=&h&\textcolor{red}{r}&\textcolor{red}{r}\\
	\gamma&0&=&\textcolor{red}{r}&\textcolor{red}{r}\\
	\alpha&0&0&=&r\\
	x&0&0&0&=
\end{pmatrix},\begin{pmatrix}
\beta&=&h&\textcolor{red}{h}&\textcolor{red}{h}\\
\gamma&0&=&\textcolor{red}{h}&\textcolor{red}{r}\\
\alpha&0&0&=&r\\
x&0&0&0&=
\end{pmatrix},\begin{pmatrix}
\beta&=&h&\textcolor{red}{h}&\textcolor{red}{r}\\
\gamma&0&=&\textcolor{red}{h}&\textcolor{red}{r}\\
\alpha&0&0&=&r\\
x&0&0&0&=
\end{pmatrix}
\end{array}\right\}\\
&=\left\{
M(\tdquatrequatre{$\beta$}{$\gamma$}{$x$}{$\alpha$}),
M(\tdquatreun{$\beta$}{$x$}{$\alpha$}{$\gamma$}),
M(\tdtroisun{$\beta$}{$\alpha$}{$\gamma$}\tdun{$x$}),
M(\tddeux{$\beta$}{$\gamma$}\tdun{$\alpha$}\tdun{$x$}),
M(\tdquatredeux{$\beta$}{$x$}{$\gamma$}{$\alpha$}),
M(\tdtroisdeux{$\beta$}{$\gamma$}{$\alpha$}\tdun{$x$})
\right\}.
\end{align*}
\end{exam}

\begin{theorem}\label{thm:dul}
Let $F,G\in \rfs$, with respectively $k$ and $l$ vertices.
\begin{align*}
F\star G&=\sum_{\sigma \in \sh(k,l)} \sum_{C\in \mathcal{FM}_\sigma(M(F),M(G))}M^{-1}(C).
\end{align*}
\end{theorem}

\begin{proof}
Let $H\in \rfs$ and $I\in N(F,G;H)$. We totally order the vertices of $H$:
\[v_1\leq_{h,r}\cdots \leq_{h,r} v_{k+l}.\]
We then put
\begin{align*}
I&=\{v_{i_1},\ldots,v_{i_k}\},&
V(H)\setminus I&=\{v_{j_1},\ldots,v_{j_l}\},
\end{align*}
with $i_1<\ldots<i_k$ and $j_1<\ldots<j_l$.
We consider the $(k,l)$ shuffle $\sigma=(i_1,\ldots,i_k,j_1,\ldots,j_l)$.
As $H_{I}=F$ and $H_{V(H)\setminus I}=G$:
\begin{enumerate}
\item If $1\leq i\leq k$, the decoration of $v_{\sigma(j)}$ is the decoration of the $i$-th vertex of $F$, so $M(H)_{\sigma(i),0}=M(F)_{i,0}$.
Similarly, if $k+1\leq i\leq k+l$,
 $M(H)_{\sigma(i),0}=M(G)_{i-k,0}$.
\item If $1\leq i<j\leq k$, then $v_{\sigma(i)}\leq_h v_{\sigma(j)}$ in $H$ if, and only if $v_i\leq_h v_j$ in $F$. So $M(H)_{\sigma(i),\sigma(j)}=M(F)_{i,j}$.
Similarly, if $k+1\leq i,j\leq k+l$, $M(H)_{\sigma(i),\sigma(j)}=M(G)_{i-k,j-k}$.
\end{enumerate}
So $M(H) \in \mathcal{FM}_\sigma(M(F),M(G))$.
Conversely, if $\sigma \in \sh(k,l)$ and $M(H)\in \mathcal{FM}_\sigma(M(F),M(H))$, then, by definition,
\begin{align*}
H_{\{\sigma(1),\ldots,\sigma(k)\}}&=F,&
H_{\{\sigma(k+1),\ldots,\sigma(k+l)\}}&=G.
\end{align*}
Therefore, we obtain that
\[|N(F,G;H)|=|\{\sigma \in \sh(k,l)\mid M(H)\in \mathcal{FM}_\sigma(M(F),M(G))\}|.\]
We obtain
\begin{align*}
F\star G&=\sum_{H\in \rfs} |N(F,G;H)| H\\
&=\sum_{H\in \rfs}|\{\sigma \in \sh(k,l)\mid  M(H)\in \mathcal{FM}_\sigma(M(F),M(G))\}| H\\
&=\sum_{\sigma \in \sh(k,l)} \sum_{C\in \mathcal{FM}_\sigma(M(F),M(G))} M^{-1}(C). \qedhere
\end{align*}
\end{proof}

\begin{exam}
From Examples \ref{examproduct1} and \ref{examproduct2}, we obtain
\begin{align*}
\tdun{$\alpha$}\star
\tddeux{$\beta$}{$\gamma$}
&=\tdtroisdeux{$\alpha$}{$\beta$}{$\gamma$}+\tdun{$\alpha$}\tddeux{$\beta$}{$\gamma$}+
\tdtroisdeux{$\beta$}{$\alpha$}{$\gamma$}
+\tdtroisun{$\beta$}{$\gamma$}{$\alpha$}
+\tdtroisdeux{$\beta$}{$\gamma$}{$\alpha$}
+\tdtroisun{$\beta$}{$\alpha$}{$\gamma$}
+\tddeux{$\beta$}{$\gamma$}\tdun{$\alpha$},\\
	\tdun{$x$}\star
\tddeux{$\beta$}{$\gamma$}
&=\tdun{$x$}\tddeux{$\beta$}{$\gamma$}+\tdtroisun{$\beta$}{$\gamma$}{$x$}
+\tdtroisdeux{$\beta$}{$\gamma$}{$x$}
+\tdtroisun{$\beta$}{$x$}{$\gamma$}
+\tddeux{$\beta$}{$\gamma$}\tdun{$x$},\\
\tdun{$\alpha$}\tdun{$x$}\star\tddeux{$\beta$}{$\gamma$}&=
\tdun{$\alpha$}\tdun{$x$}\tddeux{$\beta$}{$\gamma$}+\tdun{$\alpha$} \tdtroisun{$\beta$}{$\gamma$}{$x$}+\tdtroisdeux{$\alpha$}{$\beta$}{$\gamma$}\tdun{$x$}+\tdun{$\alpha$}\tddeux{$\beta$}{$\gamma$}\tdun{$x$}+\tdun{$\alpha$}\tdtroisun{$\beta$}{$x$}{$\gamma$}+\tdun{$\alpha$} \tdtroisdeux{$\beta$}{$\gamma$}{$x$}
+\tdquatreun{$\beta$}{$\gamma$}{$x$}{$\alpha$}
\tdtroisdeux{$\beta$}{$\alpha$}{$\gamma$}\tdun{$x$}+
\tdquatretrois{$\beta$}{$\gamma$}{$x$}{$\alpha$}\\
&+
\tdquatredeux{$\beta$}{$x$}{$\alpha$}{$\gamma$}+
\tdquatreun{$\beta$}{$x$}{$\gamma$}{$\alpha$}+
\tdquatrequatre{$\beta$}{$\gamma$}{$x$}{$\alpha$}+
\tdquatreun{$\beta$}{$x$}{$\alpha$}{$\gamma$}+
\tdtroisun{$\beta$}{$\alpha$}{$\gamma$}\tdun{$x$}+
\tddeux{$\beta$}{$\gamma$}\tdun{$\alpha$}\tdun{$x$}+
\tdquatredeux{$\beta$}{$x$}{$\gamma$}{$\alpha$}+
\tdtroisdeux{$\beta$}{$\gamma$}{$\alpha$}\tdun{$x$}.
\end{align*}\end{exam}

\begin{remark}
It is also possible to dualise the coproduct $\Delta_\lambda$ if $\lambda \neq 0$. For this, shuffles are replaced by quasi-shuffles, that is to say surjective maps $\sigma:\{1,\ldots,k+l\}\longrightarrow \{1,\ldots,n\}$, such that
\begin{align*}
	&\sigma(1)<\cdots<\sigma(k),&&\sigma(k+1)<\cdots<\sigma(k+l).
\end{align*}
The set of $(k,l)$-quasi-shuffles is denoted by $\qsh(k,l)$.
If $A$ and $B$ are two forest-representable matrices, with respectively $k$ and $l$ rows, and $\sigma \in \qsh(k,l)$, we denote by $\mathcal{FM}_\sigma(A,B)$ the set of forest-representable matrices with $n=\max(\sigma)$ rows such that:
\begin{itemize}
\item If $1\leq i\leq k$, $C_{\sigma(i),0}=A_{i,0}$.
\item If $k+1\leq i\leq k+l$, $C_{\sigma(i),0}=B_{i-k,0}$.
\item If $1\leq i,j\leq k$, $C_{\sigma(i),\sigma(j)}=A_{i,j}$.
\item If $k+1\leq i,j\leq k+l$, $C_{\sigma(i),\sigma(j)}=B_{i-k,j-k}$.
\item if $1\leq i\leq k$ and $k+1\leq j\leq k+l$, with $\sigma(i)=\sigma(j)$, then $A_{i,0}=B_{j-k,0}\in X$.
\end{itemize}
Then
\begin{align*}
F\star_\lambda G&=\sum_{\sigma \in \qsh(k,l)} \sum_{C\in \mathcal{FM}_\sigma(M(F),M(G))}\lambda^{k+l-\max(\sigma)}M^{-1}(C).
\end{align*}
For example, if $x\in X$ and $\alpha \in \Omega$,
\begin{align*}
\tdun{$x$}\star_\lambda \tddeux{$\alpha$}{$x$}
&=\tdun{$x$}\tddeux{$\alpha$}{$x$}+\tdtroisun{$\alpha$}{$x$}{$x$}
+\tdtroisdeux{$\alpha$}{$x$}{$x$}
+\tdtroisun{$\alpha$}{$x$}{$x$}
+\tddeux{$\alpha$}{$x$}\tdun{$x$}+\lambda \tddeux{$\alpha$}{$x$}=\tdun{$x$}\tddeux{$\alpha$}{$x$}+2\tdtroisun{$\alpha$}{$x$}{$x$}
+\tdtroisdeux{$\alpha$}{$x$}{$x$}
+\tddeux{$\alpha$}{$x$}\tdun{$x$}+\lambda \tddeux{$\alpha$}{$x$}.
\end{align*}
\end{remark}

\section{Hopf algebras and Rota-Baxter operators on  rooted forests}\label{sec:idc}
 We now show that there is a cocommutative Hopf algebra structure on decorated planar rooted forests under the special case of $\lambda=0$.
In this special case, we have the following results.

\subsection{Moerdijk Hopf algebras on decorated plan rooted forests}

\begin{theorem} \mlabel{cor:hopf}
$(\hrts, \conc, \etree, \col, \counit)$ is a connected graded cocommutative Hopf algebra.
\end{theorem}

\begin{proof}
We  know that $\hrts=\mathop{\oplus} \limits_{n \geq 0} \hrts_n$ with $\hrts_0 \cong \bfk$ is a graded algebra. Moreover, by Eq.~(\mref{eq:comcoproduct}), $\hrts$ is a graded connected bialgebra with the coproduct $\col$. By Lemma~2 of~\mcite{Fo3}, $\hrts$ is a Hopf algebra.

Furthermore, by Eq.~(\mref{eq:comcoproduct}) and the fact that the set of all subsets of a set equals the set of all complements of subsets of a set, $\hrts$ is a cocommutative Hopf algebra.
\end{proof}

Since $(\hrts, \conc, \etree, \col, \counit)$ is a connected graded Hopf algebra, it has an antipode $S$. Now we give a combinatorial description of $S$.

\begin{theorem}\mlabel{thm:comantipode}
For any decorated planar rooted forest $F \in \rfs\setminus\{1\}$,
\begin{align}
S(F)=\sum \limits_{I_1 \sqcup \cdots \sqcup I_k=V(F)}(-1)^k F_{I_1} \cdots F_{I_k}.
\mlabel{eq:comantipode}
\end{align}
\end{theorem}

%

\begin{proof}
By Theorem \mref{thm:comcoproduct}, the coproduct of $\hrts$ is given by
\begin{align*}
&\forall F\in \rfs\setminus\{1\},& \Delta(F)&=F\otimes 1+1\otimes F+\underbrace{\sum_{I_1\sqcup I_2=V(F)} F_{\mid I_1}\otimes F_{\mid I_2}}_{:=\tilde{\Delta}(F)}.
\end{align*}
The coproduct $\tilde{\Delta}$ defined in this way is also coassociative, and a direct induction proves that
for any $n\geq 1$, then $n-1$-iteration of $\tilde{\Delta}$ is given by
\begin{align*}
&\forall F\in \rfs\setminus\{1\},& \tilde{\Delta}^{(n-1)}(F)&=\sum \limits_{I_1 \sqcup \cdots \sqcup I_k=V(F)} F_{I_1} \otimes \cdots\otimes F_{I_k}.
\end{align*}
By Takeuchi's formula \cite{Takeuchi71}, the antipode is given by
\begin{align*}
&\forall F\in \rfs\setminus\{1\},& S(F)&=\sum_{k=1}^\infty (-1)^k m^{(k-1)}\circ \tilde{\Delta}^{(k-1)}(F)\\
&&&=\sum \limits_{I_1 \sqcup \cdots \sqcup I_k=V(F)}(-1)^k F_{I_1} \cdots F_{I_k}. \qedhere
\end{align*}
\end{proof}

\begin{exam}
By Theorem~\mref{thm:comantipode}, we have
\begin{align*}
&\ S(\tdun{$x$}\tddeux{$\alpha$}{$y$})\\
=&\ (-1)^1(F_{\{x,\alpha,y\}})+(-1)^2(F_{\{x,\alpha\}}F_{\{y\}}+F_{\{x,y\}}F_{\{\alpha\}}+F_{\{\alpha,y\}}F_{\{x\}}+F_{\{\alpha\}}F_{\{x,y\}}
+F_{\{y\}}F_{\{x,\alpha\}}+F_{\{x\}}F_{\{\alpha,y\}})\\
&\ +(-1)^3(F_{\{x\}}F_{\{\alpha\}}F_{\{y\}}+F_{\{x\}}F_{\{y\}}F_{\{\alpha\}}+F_{\{\alpha\}}F_{\{x\}}F_{\{y\}}+F_{\{\alpha\}}F_{\{y\}}F_{\{x\}}
+F_{\{y\}}F_{\{x\}}F_{\{\alpha\}}+F_{\{y\}}F_{\{\alpha\}}F_{\{x\}})\\
=&\ \tddeux{$\alpha$}{$y$}\tdun{$x$}-\tdun{$\alpha$}\tdun{$y$}\tdun{$x$}-\tdun{$y$}\tdun{$\alpha$}\tdun{$x$},\\
&\ S(\tdtroisun{$\alpha$}{$\beta$}{$x$})\\
=&\ (-1)^1(F_{\{x,\alpha,\beta\}})+(-1)^2(F_{\{x,\alpha\}}F_{\{\beta\}}+F_{\{x,\beta\}}F_{\{\alpha\}}+F_{\{\alpha,\beta\}}F_{\{x\}}
+F_{\{\alpha\}}F_{\{x,\beta\}}+F_{\{\beta\}}F_{\{x,\alpha\}}+F_{\{x\}}F_{\{\alpha,\beta\}})\\
&\ +(-1)^3(F_{\{x\}}F_{\{\alpha\}}F_{\{\beta\}}+F_{\{x\}}F_{\{\beta\}}F_{\{\alpha\}}+F_{\{\alpha\}}F_{\{x\}}F_{\{\beta\}}
+F_{\{\alpha\}}F_{\{\beta\}}F_{\{x\}}+F_{\{\beta\}}F_{\{x\}}F_{\{\alpha\}}+F_{\{\beta\}}F_{\{\alpha\}}F_{\{x\}})\\
=&\ -\tdtroisun{$\alpha$}{$\beta$}{$x$}+\tddeux{$\alpha$}{$x$}\tdun{$\beta$}+\tddeux{$\alpha$}{$\beta$}\tdun{$x$}
+\tdun{$\beta$}\tddeux{$\alpha$}{$x$}+\tdun{$x$}\tddeux{$\alpha$}{$\beta$}-\tdun{$x$}\tdun{$\alpha$}\tdun{$\beta$}-\tdun{$\alpha$}\tdun{$\beta$}\tdun{$x$}
-\tdun{$\beta$}\tdun{$x$}\tdun{$\alpha$}-\tdun{$\beta$}\tdun{$\alpha$}\tdun{$x$}.
\end{align*}
\end{exam}

\begin{remark}
When $\lambda \neq 0$ and $X$ is nonempty, choosing an $x\in X$,
\[\Delta(1+\lambda \tdun{$x$})=(1+\lambda \tdun{$x$})\otimes (1+\lambda \tdun{$x$}).\]
We obtain that $\hrts$ has non invertible group-like elements, so is not a Hopf algebra.
\end{remark}

\subsection{Rota-Baxter operator on Moerdijk Hopf algebras}

In this subsection, we show that the antipode is a Rota-Baxter operator on Moerdijk Hopf algebras.

\begin{defn}\cite{Gon21} Let $H$ be a cocommutative Hopf algebra.  A coalgebra homomorphism $\frakB:H\to H$ is called a {\bf Rota-Baxter operator on $H$} if
\begin{equation}
\frakB(a)\frakB(b)=\frakB\Big(a_1\frakB(a_2)bS(\frakB(a_3) )\Big), \quad \forall a, b\in H.
\mlabel{eq:rbog}
\end{equation}
\end{defn}

\begin{lemma}~\cite[Corollary 1]{Gon21}
Let $H$ be a cocommutative Hopf algebra with the antipode $S$. Then $S$ is a Rota-Baxter operator on $H$. \label{lemm:cH}
\end{lemma}

\begin{prop}\mlabel{prop:comantipode}
The antipode $S$ on $\hrts$ is a Rota-Baxter operator on the cocommutative Hopf algebra $\hrts$.
\end{prop}
\begin{proof}
It directly follows Theorem~\mref{thm:comantipode} and Lemma~\mref{lemm:cH}.
\end{proof}

\section{Free $\Omega$-cocycle bialgebras}\mlabel{sec:icw}

The main goal of this subsection is to understand Hopf algebras from the point view of operated algebras. This leads to the natural definition of an $\Omega$-cocycle  bialgebra.
We show that $\hrts$ is a free object in one of such categories.

\subsection{Operated algebras and free operated algebras}

\begin{defn}\cite{Guo09} Let $\Omega$ be a set.
\begin{enumerate}
\item An {\bf $\Omega$-operated monoid } is a monoid $M$ together with a set of operators $P_{\omega}: M\to M$, $\omega\in \Omega$.

\item An {\bf $\Omega$-operated unitary algebra } is a unitary algebra $A$ together with a set of linear operators $P_{\omega}: A\to A$, $\omega\in \Omega$.
\end{enumerate}
\end{defn}

When $\Omega$ is a singleton set, we will omit it.
From now on, our discussion takes place on $\hrts= \bfk \rfs$. The following results generalizes the universal properties which were studied in~\cite{CK98, Fo3, Guo09, Moe01, ZGG16}.

\begin{lemma}\cite{ZCGL19}
Let $j_{X}: X\hookrightarrow \rfs$, $x \mapsto \bullet_{x}$ be the natural embedding and $m_{\RT}$ be the concatenation product. Then, we have the following.
\begin{enumerate}
\item
The quadruple $(\rfs, \,\mul,\, 1, \, \{B_{\omega}^+\mid \omega\in \Omega\})$ together with the $j_X$ is the free $\Omega$-operated monoid on $X$.
\mlabel{it:fomonoid}
\item
The quadruple $(\hrts, \,\mul,\,1, \, \{B_{\omega}^+\mid \omega\in \Omega\})$ together with the $j_X$ is the free $\Omega$-operated unitary algebra on $X$.
\mlabel{it:fualg}
\end{enumerate}
\mlabel{lem:free}
\end{lemma}

\subsection{Free $\Omega$-cocycle Hopf algebras}

\begin{defn}
\begin{enumerate}
\item  An  {\bf $\Omega$-operated  bialgebra} (resp. {\bf Hopf algebra})  is a bialgebra (resp. Hopf algebra) $H$ together with a set of linear operators $P_{\omega}: H\to H$, $\omega\in \Omega$.
\mlabel{it:def1}
\item Let $(H,\, \{P_{\omega}\mid \omega \in \Omega\})$ and $(H',\,\{P'_{\omega}\mid \omega\in \Omega\})$ be two $\Omega$-operated bialgebras (resp. Hopf algebra). A linear map $\phi : H \rightarrow H'$ is called an {\bf $\Omega$-operated bialgebra} (resp. {\bf Hopf algebra}) {\bf morphism} if $\phi$ is a morphism of bialgebras (resp. Hopf algebras) and $\phi \circ P_\omega = P'_\omega \circ\phi$ for $\omega\in \Omega$.
\end{enumerate}
\end{defn}


%

By a  symmetric 1-cocycle condition, we then propose

\begin{defn}
\begin{enumerate}
\item \mlabel{it:def3}
An {\bf $\Omega$-cocycle bialgebra} is an $\Omega$-operated bialgebra $(H,\,m,\,1_H, \Delta,\varepsilon, \{P_{\omega}\mid \omega \in \Omega\})$ satisfying the following conditions:
\begin{align}
 \Delta P_{\omega} = (P_{\omega} \otimes \id)\Delta + (\id\otimes P_{\omega}) \Delta \quad \text{ for all }\,\, \omega\in \Omega.
\mlabel{eq:eqiterated}
\end{align}
Moreover, if $H$ is a Hopf algebra, then it is called a {\bf $\Omega$-cocycle Hopf algebra}.

\item The {\bf free $\Omega$-cocycle bialgebra on a set $X$} is an $\Omega$-cocycle bialgebra $(H_{X},\,m_{X}, \,1_{H_X}, \Delta_{X}, \,\{P_{\omega}\mid \omega \in \Omega\})$ together with a set map $j_X: X \to H_{X}$  with the property that,
for any $\Omega$-cocycle bialgebra $(H,\,m_H,\,1_H, \Delta_H,\,\{P'_{\omega}\mid \omega \in \Omega\})$ and any set map
$f: X\to H$ such that for any $x\in X$,
$$\Delta_H (f(x))=1_H \ot f(x)+f(x) \ot 1_H+\lambda f(x) \ot f(x),$$
there is a unique $\Omega$-operated bialgebra morphism $\free{f}:H_X\to H$ such that $\free{f} j_X=f$. The {\bf free $\Omega$-cocycle Hopf algebra on a set $X$} is defined similarly in the category of $\Omega$-cocycle Hopf algebras.
\mlabel{it:def4}
\end{enumerate}
\mlabel{defn:xcobi}
\end{defn}
\begin{prop}
Let $(H,\, m,\,1_H,\, \Delta,\,\varepsilon,\, \Po)$
be an $\Omega$-cocycle bialgebra. Then for each $\omega\in \Omega$,
$$P_{\omega} (H) := \{P_{\omega} (h) \mid h\in H\}$$ is an operated coideal of $H$.
\mlabel{lem:coidealv}
\end{prop}

\begin{proof}
Let $\omega\in \Omega$. We first show $P_{\omega} (h)\subseteq \ker \varepsilon$.
\begin{align*}
\varepsilon_H(h')&=\varepsilon_H\left(\sum_{(h')} h'_{(1)} \varepsilon_H( h'_{(2)})\right)
=\sum_{(h')} \varepsilon_H(h'_{(1)})  \varepsilon( h'_{(2)})
=(\varepsilon_H\ot \varepsilon_H) \Delta_H(h').
\end{align*}
Thus
\begin{align*}
\varepsilon_H(h')&=\varepsilon_H P_{\omega}(h)=(\varepsilon_H \ot_H \varepsilon_H) \Delta_H(h')=(\varepsilon_H \ot \varepsilon_H ) \Delta_H (P_{\omega}(h))\\
&=(\varepsilon_H \ot \varepsilon_H )\Big((P_{\omega} \ot \id+\id \ot P_{\omega})\Delta_H (h)\Big)\quad(\text{by Eq.~(\mref{eq:eqiterated})})\\
&=(\varepsilon_H P_{\omega }\ot \varepsilon_H )\Delta_H(h)+(\varepsilon_H \ot \varepsilon_H P_{\omega})\Delta_H(h)\\
&=\sum_{(h)}\varepsilon_H P_{\omega}(h_{(1)})\varepsilon_H (h_{(2)})+\sum_{(h)}\varepsilon_H (h_{(1)})\varepsilon_H P_{\omega}(h_{(2)})\\
&=\varepsilon_H P_{\omega}\left(\sum_{(h)}h_{(1)} \varepsilon_H (h_{(2)})\right)+\varepsilon_H P_{\omega}\left(\sum_{(h)}\varepsilon_H (h_{(1)})(h_{(2)})\right)\\
&=\varepsilon_H P_{\omega}(h)+\varepsilon_H P_{\omega}(h),
\end{align*}
which implies that $\varepsilon_H P_{\omega}(h)=0$.

Secondly, for any $h\in H$,
\begin{align*}
\Delta(P_{\omega}(h))&=(P_{\omega} \ot \id )\Delta(h)+(\id \ot P_{\omega})\Delta(h) \quad \text{ (by Eq.~(\mref{eq:eqiterated}))}\\
&= (P_{\omega} \ot \id ) \left( \sum_{(h)}(h_{(1)}\ot h_{(2)}) \right)+(\id \ot P_{\omega}) \left( \sum_{(h)}(h_{(1)}\ot h_{(2)}) \right)\\
&= \sum_{(h)}P_{\omega}(h_{(1)})\ot h_{(2)}+ \sum_{(h)}h_{(1)}\ot P_{\omega}(h_{(2)}) \in P_{\omega}(H) \otimes H+ H\otimes P_{\omega}(H).
\end{align*}
Thus $P_{\omega}(H)$ is an operated coideal.
\end{proof}

\begin{theorem}\mlabel{thm:freecocycle}
Let $j_X : X \rightarrow \hrts$, $x \rightarrow \bullet_x$ be the natural embedding. Then
\begin{enumerate}
\item \mlabel{it:11} $(\hrts,\, m_{\mathrm{RF}},$ $\etree, \,\coll, \varepsilon_{\mathrm{RF}}, \{ B_{\omega}^+ \mid \omega \in \Omega \})$ is the free $\Omega$-cocycle bialgebra on $X$.

\item \mlabel{it:22} $(\hrts,\, m_{\mathrm{RF}},$ $\etree, \,\col, \varepsilon_{\mathrm{RF}}, \{ B_{\omega}^+ \mid \omega \in \Omega \})$ is the free $\Omega$-cocycle Hopf algebra on $X$ under the case of $\lambda=0$.
\end{enumerate}
\end{theorem}

\begin{proof}
(\mref{it:11}) By Theorem~\mref{thm:main}, $(\hrts, \conc, \etree, \coll, \counit)$ is a bialgebra. Moreover by Eqs.~(\mref{eq:cdbp}) and (\mref{eq:counit}), $(\hrts,\conc, \etree,$ $\coll, \counit, \{B_{\omega}^+ \mid \omega \in \Omega \})$ is a $\Omega$-cocycle bialgebra. Now we prove it is the free $\Omega$-cocycle bialgebra on $X$. Let $(H, m_H, 1_H, \Delta_H, \varepsilon_H,$ $\{P_{\omega} \mid \omega \in \Omega \})$ be a $\Omega$-cocycle bialgebra and $f : X \rightarrow H$ a set map such that
\begin{align}
\Delta_H(f(x)) = 1_H \ot f (x)+ f (x) \ot 1_H+ \lambda f(x) \ot f(x) \, \text{ for all }\, x \in X.
\mlabel{eq:condition}
\end{align}
Particularly, $(H, m_H, 1_H, \{P_{\omega} \mid \omega \in \Omega \})$ is an $\Omega$-operated algebra. By Lemma~\mref{lem:free}, there exists a unique $\Omega$-operated algebra morphism  $\free{f}: \hrts \rightarrow H$ such
that $\free{f} \circ j_X = f$. Now it remains to  check that $\free{f}$  is bialgebra morphism, that is
\begin{align*}
\Delta_H (\free{f}(F))=(\free{f} \ot \free{f}) \coll(F) \,\, \text{and} \,\, \varepsilon_H(\free{f}(F))=\varepsilon_{\mathrm{RF}}(F) \,\, \text{for all $F \in \rfs$}.
\end{align*}

Let us consider
\begin{align*}
A := \{ F\in \hrts, \, \Delta_H(\free{f}(F))=(\free{f} \ot \free{f}) \coll(F) \}.
\end{align*}
Since $\free{f}$ is an $\Omega$-operated algebra morphism,  $\coll$ and $\Delta_H$ are algebra homomorphisms of $\hrts$, $H$, respectively, we get $1 \in A$ and $A$ is a subalgebra of $\hrts$. For any $x \in X$, we have
\begin{align*}
\Delta_H(\free{f}(\bullet_{x}))&= \Delta_H(f(x))= 1_H \ot f(x) + f(x) \ot 1_H \,\, \text{(by Eq.~(\mref{eq:condition}))},\\
(\free{f} \ot \free{f}) \coll(\bullet_{x})&=(\free{f} \ot \free{f})(\bullet_{x} \ot \etree + \etree \ot \bullet_{x}+\lambda \bullet_{x} \ot \bullet_{x})= f(x) \ot 1_H+ 1_H \ot f(x)+\lambda f(x) \ot f(x).
\end{align*}
Hence $\bullet_{x} \in A$. For $F \in A$ and $\omega \in \Omega$, we have
\begin{align*}
&\ \Delta_H(\free{f}(B_{\omega}^{+}(F)))\\
=&\ \Delta_H(P_{\omega}(\free{f}(F))) \,\, \text{(by $\free{f}$ being an $\Omega$-operated algebra morphism)} \\
=&\ (P_{\omega} \ot \id+\id \ot P_{\omega}) \Delta_H(\free{f}(F)) \\
=&\ (P_{\omega} \ot \id +\id \ot P_{\omega})(\free{f} \ot \free{f}) \coll(F) \,\, \text{(by $F \in A$)}\\
=&\ (\free{f} \ot \free{f})(B_{\omega}^{+} \ot \id +\id \ot B_{\omega}^{+})\coll(F)\,\, \text{(by $\free{f}$ being an $\Omega$-operated algebra morphism)} \\
=&\ (\free{f} \ot \free{f}) \coll(B_{\omega}^{+}(F)).
\end{align*}
Thus $A$ is stable under $B_{\omega}^{+}$ for any $\omega \in \Omega$ and so $A=\hrts$. Next, we consider
\begin{align*}
B :=\{ F \in \hrts, \, \varepsilon_H(\free{f}(F))= \varepsilon_{\mathrm{RF}}(F)\}.
\end{align*}
Since $\free{f}$ is an $\Omega$-operated algebra morphism, $\varepsilon_{\mathrm{RF}}$ and  $\varepsilon_H$ are algebra homomorphisms of $\hrts, H$, respectively, we get $1 \in B$ and $B$ is a subalgebra of $\hrts$. For any $x \in X$, by Eq.~(\mref{eq:condition}), we have $\Delta(f(x))= 1_H \ot f(x)+ f(x) \ot 1_H$. By the left counicity, we obtain
\begin{align*}
(\varepsilon_H \ot \id) \Delta(f(x))= 1_\bfk \ot f(x) +\varepsilon_H(f(x)) \ot 1_H +\lambda \varepsilon_H(f(x)) \ot f(x) =\beta_{\ell}(f(x)),
\end{align*}
which implies that $\varepsilon_H(f(x))=0$. Moreover,
\begin{align*}
\varepsilon_H(\free{f}(\bullet_{x}))= \varepsilon_H(f(x))=0=\varepsilon_{\mathrm{RF}}(\bullet_{x}),
\end{align*}
so $\bullet_{x} \in B$. For $F \in B$ and $\omega \in \Omega$, we have
\begin{align*}
\varepsilon_H(\free{f}(B_{\omega}^{+}(F)))&\ = \varepsilon_H(P_{\omega}(\free{f}(F))) \,\, \text{(by $\free{f}$ being an $\Omega$-operated algebra morphism)}\\
&\ = 0 \,\, \text{(by $H$ being a $\Omega$-cocycle bialgebra)}\\
&\ =\varepsilon_{\mathrm{RF}}(B_{\omega}^{+}(F)).
\end{align*}
Hence $B$ is stable under $B_{\omega}^{+}$ for any $\omega \in \Omega$ and so $B= \hrts$. This completes the proof.

(\mref{it:22}) By Corollary~\mref{cor:hopf}, Eqs.~(\mref{eq:cdbp}) and (\mref{eq:counit}), $(\hrts,\, m_{\mathrm{RF}},$ $\etree, \,\col, \varepsilon_{\mathrm{RF}}, \{ B_{\omega}^+ \mid \omega \in \Omega \})$ is a $\Omega$-cocycle Hopf algebra. Then Item~(\mref{it:22}) follows from Item~(\mref{it:11}) and the fact that any bialgebra morphism between two Hopf algebras is a Hopf algebra morphism.
\end{proof}


Specially, take $X=\emptyset$. Then we obtain a freeness of $\hck(\emptyset, \Omega)$.

\begin{coro} \mlabel{cor:initial}
The quintuple $(\hck(\emptyset, \Omega), \,\mul,\,\etree,\, \coll,\,\{B_{\omega}^+\mid \omega\in \Omega\})$ is the free $\Omega$-cocycle bialgebra (resp. Hopf algebra) on the empty set, that is, the initial object in the category of $\Omega$-cocycle  bialgebras.
\end{coro}

\begin{proof}
It follows from Theorem~\mref{thm:freecocycle} by taking $X=\emptyset$.
\end{proof}

Taking $\Omega$ to be singleton in Corollary~\mref{cor:initial}, then all planar rooted forests are decorated by the same letter, which can be viewed as planar rooted forests without decorations. Hence, we get the following result.

\begin{coro}
Let $\mathcal{F}$ be the set of planar rooted forests without decorations.
Then the quintuple $(\bfk\calf, \,\mul,\,\etree,\, \coll,\, B^+)$ is the free $\Omega$-cocycle bialgebra (resp. Hopf algebra) on the empty set, that is, the initial object in the category of $\Omega$-cocycle bialgebras (resp. Hopf algebra).
\mlabel{coro:rt16}
\end{coro}

\begin{proof}
It follows from Corollary~\mref{cor:initial}  by taking $\Omega$ to be a singleton set.
\end{proof}

\begin{coro}\label{coro:uni}
Let $\lambda,\mu \in \bfk$.
The following map is a bialgebra morphism:
\begin{align*}
\phi_\lambda:\left\{\begin{array}{rcl}
(\hrts, m_{\mathrm{RF}}, \Delta_\mu)&\longrightarrow&(\hrts, m_{\mathrm{RF}}, \Delta_{\lambda\mu})\\
F&\longmapsto&\lambda^{d_X(F)}F,
\end{array}\right.
\end{align*}
where $d_X(F)$ is the number of leaves of $F$ decorated by an element of $X$.
\end{coro}

\begin{proof}
If $x\in X$,
\begin{align*}
\Delta_{\lambda\mu}(\lambda \tdun{$x$})&=
\lambda \tdun{$x$}\otimes 1+\lambda \tdun{$x$}+\lambda^2\mu \tdun{$x$}\otimes \tdun{$x$}\\
&=(\lambda \tdun{$x$})\otimes 1+1\otimes (\lambda \tdun{$x$})+\mu (\lambda \tdun{$x$})\otimes (\lambda \tdun{$x$}).
\end{align*}
Therefore, by Theorem \ref{thm:freecocycle}, there exists a unique morphism $\phi_\lambda$ of $\Omega$-cocycle bialgebras from $(\hrts,\, m_{\mathrm{RF}},$ $\etree, \,\Delta_\mu, \varepsilon_{\mathrm{RF}}, \{ B_{\omega}^+ \mid \omega \in \Omega \})$ to $(\hrts,\, m_{\mathrm{RF}},$ $\etree, \,\Delta_{\lambda\mu}, \varepsilon_{\mathrm{RF}}, \{ B_{\omega}^+ \mid \omega \in \Omega \})$,
sending $\tdun{$x$}$ to $\lambda\tdun{$x$}$ for any $x\in X$.
An easy induction on the number of vertices of $F$ proves the formula for $\phi_\lambda(F)$.  \end{proof}

\begin{remark}
The morphism $\phi_\lambda$ is an insomorphism if, and only if, $\lambda$ is invertible in $\bfk$. If so, its inverse is $\phi_{\lambda^{-1}}$.
\end{remark}

\noindent {\bf Acknowledgments}:
This research is supported by the National Natural Science Foundation of China (Grant No.\@ 12301025, 12101316 ).

\noindent
{\bf Declaration of interests.} The authors have no conflicts of interest to disclose.

\noindent
{\bf Data availability.} Data sharing is not applicable as no new data were created or analyzed.

\end{document}